\documentclass[12pt,a4paper,draft]{amsart}
\usepackage{amsmath,amsfonts,amssymb,amsthm,amscd}
\usepackage[T1]{fontenc}
\usepackage{ae,aecompl}
\usepackage[mathscr]{eucal}
\usepackage{mathrsfs}
\usepackage[english]{babel}
\usepackage{latexsym}
\usepackage[all]{xy}
\usepackage{cite}
\usepackage{tikz}
\usetikzlibrary{graphs}
\usetikzlibrary{calc}
\usepackage{graphicx}
\usepackage[english]{isodate}
\usepackage{calc,array}
\newcommand{\cellstart}{\rule{0pt}{\heightof{A}+1ex}}
\newcommand{\cellfinish}{\rule[-2ex]{0pt}{2ex}}
\usepackage{hyperref}
\hypersetup
{
    pdfauthor={Seung-Jo Jung},
    pdftitle={Terminal Quotient Singularities in Dimension three via variation of GIT}
}

\newcommand{\hookdownarrow}{\mathrel{\rotatebox[origin=c]{-90}{$\hookrightarrow$}}}

\theoremstyle{plain}
\newtheorem{Thm}[equation]{Theorem}
\newtheorem{Prop}[equation]{Proposition}
\newtheorem{Cor}[equation]{Corollary}
\newtheorem{Lem}[equation]{Lemma}
\newtheorem{Con}[equation]{Conjecture}

\theoremstyle{definition}

\newtheorem{Def}[equation]{Definition}

\newtheorem{Eg}[equation]{Example}

\newtheorem{Rem}[equation]{Remark}
\numberwithin{equation}{section}
\numberwithin{figure}{section}
\numberwithin{table}{section}

\newcommand{\A}{\mathbb A} 
\newcommand{\C}{\mathbb C} 
 
\newcommand{\Q}{\mathbb Q} 
\newcommand{\R}{\mathbb R} 
\newcommand{\T}{\mathbf T} 
\newcommand{\Z}{\mathbb Z} 
\newcommand{\Zp}{\mathbb{Z}_{\geq 0}} 

\newcommand{\bu}{\mathbf u} 
\newcommand{\bbf}{\mathbf f} 
\newcommand{\bg}{\mathbf g} 
\newcommand{\bk}{\mathbf k} 
\newcommand{\bm}{\mathbf m} 
\newcommand{\bn}{\mathbf n} 

\newcommand{\bb}{\mathbf b} 

\newcommand{\Cone}{\ensuremath{\operatorname{Cone}}}

\newcommand{\eqd}{\stackrel{\rm \tiny def}{=}} 

\newcommand{\sF}{\mathcal F} 
\newcommand{\sG}{\mathcal G} 
\newcommand{\sM}{\mathcal M} 

\newcommand{\sV}{\mathcal V}

\newcommand{\OZ}{{\mathcal O}\!_{Z}} 
\newcommand{\KY}{{K}\!_{Y}} 
\newcommand{\KX}{{K}\!_{X}} 

\newcommand{\dual}{^{\vee}} 
\newcommand{\mul}{^{\times}} 

\newcommand{\AHilb}{\ensuremath{A\text{-}\operatorname{\! Hilb}}}
\newcommand{\GHilb}[1]{\ensuremath{G\text{-}\operatorname{\! Hilb}}\C^{#1}}

\newcommand{\GHii}[1]{\ensuremath{G_{#1}\text{-}\operatorname{\! Hilb}}\C^3}
\newcommand{\GHil}{\ensuremath{G\text{-}\operatorname{\! Hilb}}}

\newcommand{\Spec}{\ensuremath{\operatorname{Spec}}}

\newcommand{\Hom}{\ensuremath{\operatorname{Hom}}}
\DeclareMathOperator{\GL}{GL}
\DeclareMathOperator{\SL}{SL}
\DeclareMathOperator{\wt}{wt} 
\DeclareMathOperator{\Irr}{Irr} 

\newcommand{\rg}{\ensuremath{\operatorname{Rep}}G} 
\newcommand{\rgssth}{\ensuremath{\operatorname{Rep}}^{ss}_{\theta}G}

\DeclareMathOperator{\D}{D} 

\newcommand{\lgr}{\frac{1}{a}(1,-r,r)}
\newcommand{\rgr}{\frac{1}{r-a}(1,r,-r)}

\newcommand{\diag}{\ensuremath{\operatorname{diag}}}

\newcommand{\git}{\ensuremath{\operatorname{/\!\!/}}}
\newcommand{\iso}{\cong}

\newcommand{\st}{\ensuremath{\operatorname{\bigm{|}}}}
\newcommand{\stB}{\ensuremath{\operatorname{\bigg{|}}}}
\newcommand{\stb}{\ensuremath{\operatorname{\Big{|}}}}

\newcommand{\QED}{\hfill$\blacksquare$\end{proof}\vskip 3pt}
\newcommand{\eeg}{~\hfill$\lozenge$\end{Eg}\vskip 3pt }
\newcommand{\erem}{~\hfill$\lozenge$\end{Rem}\vskip 3pt }
\newcommand{\ewar}{~\hfill$\lozenge$\end{War}\vskip 3pt }

\newcommand{\pull}{^{\ast}}

\newcommand{\inv}{^{-1}}

\newcommand{\hh}{\ensuremath{\operatorname{H}}} 

\newcommand{\mon}{\overline{M}_{\geq 0}} 
\newcommand{\lau}{\overline{M}} 
\newcommand{\wtga}{\wt_{\Gamma}} 
\newcommand{\wtgp}{\wt_{\Gamma'}}
\newcommand{\mth}{\sM_{\theta}} 
\newcommand{\mzero}{\sM_{0}} 
\newcommand{\yth}{Y_{\theta}} 

\newcommand{\lf}{\lfloor}
\newcommand{\rf}{\rfloor}

\newcommand{\one}{\boldsymbol{1}}

\newcommand{\Homz}{\ensuremath{\operatorname{Hom}}_{\mathbb Z}}
\newcommand{\zzz}{\overline{L}}

\newcommand{\ev}[1]{\ensuremath{\operatorname{ev}}_{#1}}

\newcommand{\hd}{{\mathfrak{h}\pull}} 
\newcommand{\mco}{{\Sigma_{\mathrm{max}}}} 
\newcommand{\wc}{{\mathfrak{C}}} 
\newcommand{\gr}{{\mathfrak{S}}} 
\newcommand{\sra}{{\Delta}} 
\newcommand{\wca}{{\mathfrak{C}}(r,a)} 
\newcommand{\gra}{{\gr(r,a)}} 
\newcommand{\grx}{{\gr(r,a)_0}} 
\newcommand{\gld}{\GL(\delta)} 

\newcommand{\vtheta}{\psi} 
\title[Terminal Quotient Singularities in Dimension three]{Terminal Quotient Singularities in Dimension three via variation of GIT}

\author[S.-J. Jung]{Seung-Jo Jung}
\address{Korea Institute for Advanced Study, 85 Hoegiro, Dongdaemun-gu, Seoul, 130-722, Republic of Korea}
\email{seungjo@kias.re.kr}

\keywords{the McKay correspondence, terminal quotient singularities, economic resolutions}

\subjclass[2010]{14E15, 14E16, 14L24}

\begin{document}


\begin{abstract}
A 3-fold terminal quotient singularity $X=\C^3/G$ admits the economic resolution $Y\rightarrow X$, which is ``close to being crepant". This paper proves that the economic resolution $Y$ is isomorphic to a distinguished component of a moduli space of certain $G$-equivariant objects using the King stability condition $\theta$ introduced by  K\k{e}dzierski\cite{Ked14}. 
\end{abstract}
\maketitle
\tableofcontents


\section{Introduction}\pagenumbering{arabic}\label{Sec:Introduction}
The motivation of this work stems from the philosophy of the {\em McKay correspondence}, which says that if a finite group~$G$ acts on a variety~$M$, then a crepant resolution of the quotient~$M/G$ can be realised as a moduli space of $G$-equivariant objects on~$M$.

Let $G\subset \GL_n(\C)$ be a finite group. A $G$-equivariant coherent sheaf~$\sF$ on~$\C^n$ is called a {\em $G$-constellation} if $\hh^{0}(\sF)$ is isomorphic to the regular representation $\C[G]$ of $G$ as a $\C[G]$-module. In particular, the structure sheaf of a $G$-invariant subscheme $Z \subset \C^n$ with $\hh^{0}(\OZ)$ isomorphic to $\C[G]$, which is called a {\em $G$-cluster}, is a $G$-constellation. 
Define the GIT stability parameter space
\[
\Theta = \left\{\theta \in \Hom_{\Z} (R(G),\Q) \st \theta\left(\C[G]\right) =0 \right\}\!,
\]
where $R(G)$ is the representation ring of $G$. For $\theta\in\Theta$, we say that a $G$-constellation $\sF$ is $\theta$-(semi)stable if $\theta(\sG)> 0\ (\theta(\sG)\geq 0)$ for every nonzero proper subsheaf $\sG$ of $\sF$.
A parameter $\theta$ is called {\em generic} if every $\theta$-semistable $G$-constellation is $\theta$-stable.

Let $\mth$ be the moduli space of $\theta$-semistable $G$-constellations. 
In the celebrated paper \cite{BKR}, Bridgeland, King and Reid proved that for a finite subgroup $G$ of $\SL_3(\C)$, $\mth$ is a crepant resolution of $\C^3/G$ if $\theta$ is generic. Craw and Ishii \cite{CI} showed that in the case of a finite abelian group $G \subset \SL_3(\C)$, {\em any} projective crepant resolution can be realised as $\mth$ for a suitable GIT parameter~$\theta$.

While the moduli space $\mth$ need not be irreducible\cite{CMTb} in general, Craw, Maclagan and Thomas~\cite{CMT} showed that for generic $\theta$, $\mth$ has a unique irreducible component $\yth$ containing the torus $(\C\mul)^n/G$ if $G$ is abelian. The component $\yth$ is birational to $\C^n/G$ and is called the {\em birational component\footnote{This component is also called the coherent component.}} of $\mth$.

On the other hand, in the case of $G \subset \GL_3(\C)$ giving a terminal quotient singularity $X=\C^3/G$ in dimension 3, $X$ has the {\em economic resolution} $\phi\colon Y \rightarrow X$ satisfying
\[\KY=\varphi\pull(\KX)+\sum_{1 \leq i < r}\frac{i}{r}E_i\]
with $E_i$'s prime exceptional divisors. 
K\k{e}dzierski\cite{Ked14} proved that $Y$ is isomorphic to the normalization of $\yth$ for some $\theta$.
The main theorem of this paper is that the economic resolution $Y$ of $X$ can be interpreted as a component of a moduli space of $G$-constellations as follows.
\begin{Thm}[Theorem~\ref{Thm:Main Theorem}]\label{intro_a:main theorem}
The economic resolution $Y$ of a 3-fold terminal quotient singularity $X=\C^3/G$ is isomorphic to the birational component $\yth$ of the moduli space $\mth$ of $\theta$-stable $G$-constellations for a suitable parameter~$\theta$.
\end{Thm}

To prove the theorem, first we generalize Nakamura's result~\cite{N01}. Let $G\subset \GL_n(\C)$ be a finite diagonal group.
Nakamura\cite{N01} introduced a $G$-graph which is a $\C$-basis of $\OZ$ for a torus invariant $G$-cluster $Z$. Using $G$-graphs, he described a local chart of $\GHil$. In this paper, we introduce a {\em $G$-prebrick} which is a $\C$-basis of $\hh^0(\sF)\iso \C[G]$ for a torus invariant $G$-constellation $\sF$.

For a $G$-prebrick $\Gamma$, by King\cite{K94}, we have an affine scheme $D(\Gamma)$ parametrising $G$-constellations whose basis is $\Gamma$. The affine scheme $D(\Gamma)$ is not necessarily irreducible, but $D(\Gamma)$ has a distinguished component $U(\Gamma)$ containing the torus $T=(\C\mul)^n/G$. In addition, we can show that $U(\Gamma)=\Spec \C[S(\Gamma)]$ for a semigroup $S(\Gamma)$. If the toric affine variety $U(\Gamma)$ has a torus fixed point, then $\Gamma$ is called a {\em $G$-brick}. We can prove that $\yth$ is covered by $U(\Gamma)$'s for suitable $G$-bricks $\Gamma$.

On the other hand, from \cite{MS, R87}, we know that a 3-fold quotient singularity $X= \C^3/G$ has terminal singularities if and only if the group $G$ is of type $\frac{1}{r}(1,a,r-a)$ with $r$ coprime to $a$, i.e.\ 
\[
G=\langle\diag(\epsilon,\epsilon^a,\epsilon^{r-a}) \st \epsilon^r=1\rangle.\]
In this case, the quotient variety $X=\C^3/G$ is not Gorenstein. 
While $X$ does not admit a crepant resolution, $X$ has the {\em economic resolution} $\phi\colon Y \rightarrow X$
obtained by a toric method called {\em weighted blowups} (or {\em Kawamata blowups}). For each step of the weighted blowups, we define three {\em round down functions}, which are maps between monomial lattices.

As $Y$ is toric, $Y$ is determined by its associated toric fan $\Sigma$ with the lattice $M$ of $G$-invariant monomials. From toric geometry, note that $Y$ is covered by torus invariant affine open subsets $U_{\sigma}=\Spec \C[\sigma\dual\cap M]$ for $\sigma \in \mco$ where $\mco$ denotes the set of maximal cones in $\Sigma$.

Using the round down functions, we find a set $\gr$ of $G$-bricks such that there exists a bijective map $\mco \rightarrow \gr$ sending $\sigma$ to $\Gamma_{\sigma}$ with $U(\Gamma_{\sigma})\iso U_{\sigma}$. We show that there exists a parameter $\theta\in\Theta$ such that $U(\Gamma_{\sigma})$'s cover $\yth$ for $\Gamma_{\sigma} \in \gr$. This proves that the economic resolution $Y$ is isomorphic to the birational component $\yth$ of $\mth$.

Moreover, we further prove $D(\Gamma)\iso \C^3$ for $\Gamma\in \gr$. So the irreducible component $\yth$ is actually a connected component. We conjecture that the moduli space $\mth$ is irreducible, which implies $Y \iso\mth$.

\subsubsection*{\bf Layout of this article}
In Section~\ref{Sec:G-bricks and moduli spaces of G-constellations}, we define $G$-(pre)bricks and describe the birational component $\yth$ using $G$-bricks.
Section~\ref{Sec:Wt Blups and Econ Resolns} explains how to obtain the economic resolutions using toric methods and defines round down functions. 
In Section~\ref{Sec:Main Theorem}, we explain how to find $G$-bricks and a parameter $\theta\in \Theta$ such that the economic resolution is isomorphic to the birational component $\yth$. 
In Section~\ref{Sec:Kedzierski's GIT chamber} we describe Kedzierski's GIT chamber using the $A_{r-1}$ root system. In Section~\ref{Sec:example for 1/12(1,7,5)}, we calculate $G$-bricks and Kedzierski's GIT chamber for the group of type $\frac{1}{12}(1,7,5)$.
\subsubsection*{\bf Acknowledgement} This article is part of my PhD thesis in the University of Warwick \cite{Thesis}. 
I am very grateful to my advisor Miles Reid for sharing his views on this subject and many ideas. I would like to thank Diane Maclagan, Alastair Craw, Hiraku Nakajima, and Timothy Logvinenko for valuable conversations. I thank Andrew Chan and Tom Ducat for their comments on earlier drafts.
\section{$G$-bricks and moduli spaces of $G$-constellations}\label{Sec:G-bricks and moduli spaces of G-constellations}
In this section we define a $G$-{\em prebrick} which is a generalized version of Nakamura's~$G$-graph from \cite{N01}. By using $G$-prebricks, we describe local charts of moduli spaces of $G$-constellations.

In this section, we restrict ourselves to the case where $G$ is a finite cyclic subgroup of $\GL_3(\C)$. It is possible to generalize part of the argument to include general finite small abelian groups in $\GL_n(\C)$ for any dimension $n$. However we prefer to focus on this case where we can avoid the difficulty of notation.
\subsection{Moduli spaces of $G$-constellations}\label{subsec:Moduli of G-constellations}
In this section, we review the construction of moduli spaces $\mth$ of $\theta$-stable $G$-constellations as described in \cite{K94,CI}.

Define the group $G=\langle\diag(\epsilon^{\alpha_1},\epsilon^{\alpha_2},\epsilon^{\alpha_3}) \st \epsilon^r=1 \rangle \subset \GL_3(\C)$. We call $G$ {\em the group of type $\frac{1}{r}(\alpha_1,\alpha_2,\alpha_3)$}. 
We can identify the set of irreducible representations of~$G$ with the character group $G\dual :=\Hom(G,\C\mul)$ of~$G$. Note that the regular representation~$\C[G]$ is isomorphic to~$\bigoplus_{\rho \in G\dual} \C \rho$.
\begin{Def}
A {\em $G$-constellation} on $\C^3$ is a $G$-equivariant coherent sheaf $\sF$ on $\C^3$ with $\hh^0(\sF)$ isomorphic to the regular representation~$\C[G]$ of $G$ as a $\C[G]$-module. 
\end{Def}

The representation ring $R(G)$ of $G$ is $\bigoplus_{\rho \in G\dual} \Z\cdot \rho$. 
Define the GIT stability parameter space
\[
\Theta = \left\{\theta \in \Hom_{\Z} (R(G),\Q) \st \theta\left(\C[G]\right) =0 \right\}\!.
\]
\begin{Def}
For a stability parameter $\theta \in \Theta$, we say that:
\begin{enumerate}
\item a $G$-constellation $\sF$ is {\em $\theta$-semistable} if $\theta(\sG) \geq 0$ for every proper submodule $\sG \subset \sF$.
\item a $G$-constellation $\sF$ is {\em $\theta$-stable} if $\theta(\sG) > 0$ for every nonzero proper submodule $\sG \subset \sF$.
\item $\theta$ is {\em generic} if every $\theta$-semistable object is $\theta$-stable.
\end{enumerate}
\end{Def}

It is known \cite{CMTb} that the language of $G$-constellations is the same as the language of the {\em McKay quiver representations}. Thus by King~\cite{K94}, the moduli spaces of $G$-constellations can be constructed using Geometric Invariant Theory (GIT).

Let $\rg$ be the affine scheme whose coordinate ring is
\[
\C[\rg]=\C[x_i,y_i,z_i \st  i \in G\dual ]\big{/}I_{G}
\]
where $I_{G}$ is the ideal generated by the following quadrics:
\begin{equation}\label{Eqtn:the commutativity relations on quivers}
\begin{cases}
x_{i}y_{i+\alpha_1} - y_{i}x_{i+\alpha_2},\\
x_{i}z_{i+\alpha_1} - z_{i}x_{i+\alpha_3},\\
y_{i}z_{i+\alpha_2} - z_{i}y_{i+\alpha_3}.
\end{cases}
\end{equation}
Let $\delta=(1,\ldots,1) \in \Zp^r$. The group~$\GL(\delta):=\prod_{i \in G\dual} \C\mul = (\C\mul)^r$ acts on~$\rg$ via change of basis.
For a parameter $\theta\in\Theta$, define the GIT quotient with respect to $\theta$
\[\rg \git_{{\theta}} \GL(\delta):=\rgssth /\GL(\delta)\]
parametrising closed $\GL(\delta)$-orbits in $\rgssth$ where $\rgssth$ denotes the $\theta$-semistable locus in $\rg$.

\begin{Thm}[King\cite{K94}] Let us define
$\mth:= \rg \git_{{\theta}} \GL(\delta)$.
\begin{enumerate}
\item The quasiprojective scheme $\mth$ is a coarse moduli space of $\theta$-semistable $G$-constellations up to S-equivalence. 
\item If $\theta$ is generic, the scheme $\mth$ is a fine moduli space of $\theta$-stable $G$-constellations.
\item The scheme $\mth$ is projective over $\mzero=\Spec \C[\rg]^{\GL(\delta)}$.
\end{enumerate}
\end{Thm}

\subsubsection*{\bf Birational component $\yth$ of the moduli space $\mth$} 
Let $\mth$ denote the moduli space of $\theta$-semistable $G$-constellations. Note that the moduli space $\mth$ need not be irreducible\cite{CMTb}.

Note that for every parameter $\theta$, there exists a natural embedding of the torus $T:=(\C\mul)^3/G$ into $\mth$. Indeed, for a $G$-orbit~$Z$ in the algebraic torus $\T:=(\C\mul)^3 \subset \C^3$, since $Z$ is a free $G$-orbit, $\OZ$ has no nonzero proper submodules. Thus $\OZ$ is a $\theta$-stable $G$-constellation. Hence it follows that the torus $T:=(\C\mul)^3/G$ is the fine moduli space of $\theta$-stable $G$-constellations supported on~$\T$ because any $G$-constellation supporting on a free $G$-orbit~$Z$ is isomorphic to $\OZ$.
\begin{Thm}[Craw, Maclagan and Thomas{\cite{CMT}}]\label{Thm:CMT}
Let $\theta \in \Theta$ be generic. Then $\mth$ has a unique irreducible component $\yth$ that contains the torus $T:=(\C\mul)^n/G$. Moreover $\yth$ satisfies the following properties:
\[
\begin{array}{ccc}
\yth & \hookrightarrow & \mth \\[4pt]
\Big\downarrow && \Big\downarrow \\[4pt]
\C^3/G & \hookrightarrow & \mzero
\end{array}
\]
\begin{enumerate}
\item $\yth$ is a not-necessarily-normal toric variety which is birational to the quotient variety $\C^3/G$.
\item $\yth$ is projective over the quotient variety $\C^3/G$.
\end{enumerate}
\end{Thm}
\begin{Def}
The unique irreducible component $\yth$ in Theorem~\ref{Thm:CMT} is called the {\em birational component} of $\mth$. 
\end{Def}
Since Craw, Maclagan and Thomas\cite{CMT} constructed $\yth$ as GIT quotient of a reduced affine scheme, it follows that $\yth$ is reduced.
\begin{Rem}\label{Rem:Torus action on mth}
Since $\T=(\C\mul)^3$ acts on $\C^3$, the algebraic torus $\T$ acts on the moduli space $\mth$ naturally. Fixed points of the $\T$-action play a crucial role in the study of the moduli space $\mth$. \erem
\subsection{$G$-prebricks and local charts of $\mth$}\label{subsec:G-bricks and local charts of yth}
Let $G \subset \GL_3(\C)$ be the finite group of type~$\frac{1}{r}(\alpha_1,\alpha_2,\alpha_3)$. 
Define the lattice 
\begin{equation*}
L = \Z^3 + \Z\cdot \frac{1}{r}(\alpha_1,\alpha_2,\alpha_3),
\end{equation*}
which is an overlattice of $\zzz = \Z^3$ of finite index. Let $\{e_1,e_2,e_3\}$ be the standard basis of $\Z^3$. Set $\lau=\Homz(\zzz,\Z)$ and $M = \Homz(L,\Z)$. The two dual lattices $\lau$ and $M$ can be identified with Laurent monomials and $G$-invariant Laurent monomials, respectively. 
The embedding of~$G$ into the torus $(\C\mul)^3\subset \GL_3(\C)$ induces a surjective homomorphism
\[
\wt \colon \lau \longrightarrow G\dual
\]
whose kernel is $M$. 
Note that there are two isomorphisms of abelian groups $L/\Z^3 \rightarrow G$ and $\lau / M \rightarrow G\dual$.

Let $\mon$ denote genuine monomials in $\lau$, i.e.\
\begin{equation*}
\mon=\{x^{m_1}y^{m_2}z^{m_3} \in \lau \st m_1,m_2,m_3 \geq 0 \}.
\end{equation*}
For a set $A \subset \C[x^{\pm},y^{\pm},z^{\pm}]$, let $\langle A \rangle$ denote the $\C[x,y,z]$-submodule of $\C[x^{\pm},y^{\pm},z^{\pm}]$ generated by $A$.

Let $\sigma_{+}$ be the cone in $L_{\R}:=L \otimes_{\Z} \R $ generated by $e_1,e_2,e_3$.
Note that the corresponding affine toric variety 
$U_{\sigma_{+}}=\Spec \C[\sigma_+\dual \cap M]$ is isomorphic to the quotient variety $\C^3/G=\Spec \C[x,y,z]^G$.
%
%
\begin{Def}\label{Def:G-prebrick}\index{$G$-prebrick}
A {\em $G$-prebrick} $\Gamma$ is a subset of Laurent monomials in $\C[x^{\pm},y^{\pm},z^{\pm}]$ satisfying:
\begin{enumerate}
\item the monomial $\one$ is in $\Gamma$.
\item for each weight $\rho \in G\dual$, there exists a unique Laurent monomial $\bm_{\rho}\in \Gamma$ of weight $\rho$, i.e.\ $\wt\colon\Gamma \rightarrow G\dual$ is bijective.
\item if $\bn' \cdot \bn \cdot \bm_{\rho} \in \Gamma$ for $\bm_{\rho}\in \Gamma$ and $\bn, \bn' \in \mon$, then $\bn \cdot \bm_{\rho} \in \Gamma$.
\item the set $\Gamma$ is {\em connected} in the sense that for any element $\bm_{\rho}$, there is a (fractional) path in $\Gamma$ from $\bm_{\rho}$ to $\one$  whose steps consist of multiplying or dividing by one of $x,y,z$.
\end{enumerate}
\end{Def}
For a Laurent monomial $\bm \in \lau$, let $\wtga(\bm)$ denote the unique element $\bm_{\rho}$ in $\Gamma$ of the same weight as $\bm$. 
\begin{Rem}
Nakamura's $G$-graph $\Gamma$ in\cite{N01} is a $G$-prebrick
 because if a monomial $\bn' \cdot \bn$ is in $\Gamma$ for two monomials $\bn, \bn' \in \mon$, then $\bn$ is in $\Gamma$. 
The main difference between $G$-graphs and $G$-prebricks is that elements of $G$-prebricks are allowed to be Laurent monomials, not just genuine monomials.\erem
\begin{Eg} \label{Eg:G graph 1/7(1,3,4) }
Let $G$ be the group of type $\frac{1}{7}(1,3,4)$. Then 
\begin{align*}
\Gamma_1&=\Big\{1,y,y^2,z,\tfrac{z}{y},\tfrac{z^2}{y},\tfrac{z^2}{y^2}\Big\}, \\[4pt]
\Gamma_2&=\Big\{1,z,y,y^2,\tfrac{y^2}{z},\tfrac{y^3}{z},\tfrac{y^3}{z^2}\Big\}
\end{align*}
are $G$-prebricks.
For $\Gamma_1$, we have $\wt_{\Gamma_1}(x)=\tfrac{z}{y}$ and $\wt_{\Gamma_1}(y^3)=\tfrac{z^2}{y^2}$.
\eeg
For a $G$-prebrick $\Gamma=\{\bm_{\rho}\}$, as an analogue of \cite{N01},  define $S(\Gamma)$ to be the subsemigroup of $M$ generated by $\dfrac{\bn \cdot \bm_{\rho}}{\wtga(\bn \cdot \bm_{\rho})}$ for all $\bn \in \mon$, $\bm_{\rho}\in \Gamma$. Define a
cone $\sigma(\Gamma)$ in $L_{\R}=\R^3$ as follows:
\begin{align*}
\sigma(\Gamma)
&=S(\Gamma)\dual\\
&=\left\{ \bu \in L_{\R} \stB  \left\langle \bu, \frac{\bn \cdot \bm_{\rho}}{\wtga(\bn \cdot \bm_{\rho})}\right\rangle \geq 0, \quad \forall \bm_{\rho}\in \Gamma, \ \bn \in \mon \right\}\!.
\end{align*}
Observe that:
\begin{enumerate}
\item $\big(\mon \cap M \big)\subset S(\Gamma)$,
\item $\sigma(\Gamma) \subset \sigma_{+}$,
\item $S(\Gamma) \subset \big(\sigma(\Gamma)\dual \cap M\big)$.
\end{enumerate}
\begin{Lem}\label{Lem:crucial finitely gen}
Let $\Gamma$ be a $G$-prebrick. Define
\[
B(\Gamma) := \big\{ \bbf \cdot \bm_{\rho} \st \bm_{\rho} \in \Gamma, \ \bbf \in\{x,y,z\} \big\} \!\setminus \!\Gamma.
\]
Then the semigroup $S(\Gamma)$ is generated by $\frac{\bb}{\wtga(\bb)}$ for all $\bb \in B(\Gamma)$ as a semigroup. In particular, $S(\Gamma)$ is finitely generated as a semigroup.
\end{Lem}
\begin{proof}
Let $S$ be the subsemigroup of $M$ generated by $\frac{\bb}{\wtga(\bb)}$ for all $\bb \in B(\Gamma)$. Clearly, $S \subset S(\Gamma)$. For the opposite inclusion, it is enough to show that the generators of $S(\Gamma)$ are in $S$.

An arbitrary generator of $S(\Gamma)$ is of the form $\frac{\bn \cdot \bm_{\rho}}{\wtga(\bn \cdot \bm_{\rho})}$ for some $\bn \in \mon$, $\bm_{\rho}\in \Gamma$. We may assume that $\bn \cdot \bm_{\rho} \not\in \Gamma$. In particular, $\bn \neq \one$. Since $\bn$ has positive degree, there exists $\bbf \in \{x,y,z\}$ such that $\bbf$ divides $\bn$, i.e.\ $\frac{\bn}{\bbf} \in \mon$ and $\deg(\frac{\bn}{\bbf}) < \deg(\bn)$. Let $\bm_{\rho'}$ denote $\wtga(\frac{\bn}{\bbf} \cdot \bm_{\rho})$. Note that \[
\wtga(\bbf \cdot \bm_{\rho'})=\wtga(\bbf \cdot \frac{\bn}{\bbf} \cdot \bm_{\rho})=\wtga(\bn \cdot \bm_{\rho}).
\]
Thus
\begin{align*}
\frac{\bn \cdot \bm_{\rho}}{\wtga(\bn \cdot \bm_{\rho})} 
&= \frac{\frac{\bn}{\bbf} \cdot \bm_{\rho}}{\wtga(\frac{\bn}{\bbf} \cdot \bm_{\rho})} \cdot \frac{\bbf \cdot \wtga(\frac{\bn}{\bbf} \cdot \bm_{\rho})}{\wtga(\bn \cdot \bm_{\rho})}\\
&=\frac{\frac{\bn}{\bbf} \cdot \bm_{\rho}}{\wtga(\frac{\bn}{\bbf} \cdot \bm_{\rho})} \cdot \frac{\bbf \cdot \bm_{\rho'}}{\wtga(\bbf \cdot \bm_{\rho'})}.
\end{align*}
By induction on the degree of monomial $\bn$, the assertion is proved.
\end{proof}
The set $B(\Gamma)$ in the lemma above is called the {\em Border bases} of $\Gamma$. As $B(\Gamma)$ is finite, the semigroup $S(\Gamma)$ is finitely generated as a semigroup. Thus the semigroup~$S(\Gamma)$ defines an affine toric variety.
Define two affine toric varieties:
\begin{align*}
U(\Gamma) &:= \Spec \C[S(\Gamma)],\\
U^{\nu}(\Gamma)&:= \Spec \C[\sigma(\Gamma)\dual\cap M]. 
\end{align*}
Note that the torus $\Spec \C[M]$ of $U(\Gamma)$ is isomorphic to $T=(\C\mul)^3/G$ and that $U^{\nu}(\Gamma)$ is the normalization of $U(\Gamma)$.

Craw, Maclagan and Thomas\cite{CMTb} showed that there exists a torus invariant $G$-cluster which does not lie in the birational component~$\yth$.
The following definition is implicit in \cite{CMTb}.
\begin{Def}
A $G$-prebrick $\Gamma$ is called a {\em $G$-brick} if the affine toric variety $U(\Gamma)$ contains a torus fixed point.
\end{Def}
From toric geometry, $U(\Gamma)$ has a torus fixed point if and only if $S(\Gamma) \cap (S(\Gamma))\inv=\{\one\}$, i.e.\ the cone $\sigma(\Gamma)$ is a 3-dimensional cone.
\begin{Eg} \label{Eg:cone assoc to G graph 1/7(1,3,4) }
Consider the $G$-prebricks $\Gamma_1, \Gamma_2$ in Example~\ref{Eg:G graph 1/7(1,3,4) }. By Lemma~\ref{Lem:crucial finitely gen}, $S(\Gamma_1)$ is generated by $\tfrac{y^5}{z^2},\tfrac{z^3}{y^4},\tfrac{xy}{z}$. We have
\begin{align*}
\sigma(\Gamma_1)&=\left\{ \bu \in \R^3 \stb \langle \bu, \bm \rangle \geq 0, \text{ for all $\bm\in\big\{\tfrac{y^5}{z^2},\tfrac{z^3}{y^4},\tfrac{xy}{z}\big\}$}\right\}\! , \\[3pt]
&=\Cone\left((1,0,0),\tfrac{1}{7}(3,2,5),\tfrac{1}{7}(1,3,4)\right).
\end{align*}
Similarly, we can see that
\begin{align*}
\sigma(\Gamma_2)&=\left\{ \bu \in \R^3 \stb  \langle \bu, \bm \rangle \geq 0, \text{ for all $\bm\in\big\{\tfrac{y^4}{z^3},\tfrac{z^4}{y^3},\tfrac{xz^2}{y^3}\big\}$}\right\}\! ,\\[3pt]
&=\Cone\left((1,0,0),\tfrac{1}{7}(1,3,4),\tfrac{1}{7}(6,4,3)\right).
\end{align*}
Since $S(\Gamma_1)=\sigma(\Gamma_1)\dual \cap M  $ and $S(\Gamma_2) = \sigma(\Gamma_2)\dual \cap M $, the two $G$-prebricks $\Gamma_1$, $\Gamma_2$ are $G$-bricks. Moreover the two toric varieties $U(\Gamma_1)$ and $U(\Gamma_2)$ are smooth.\eeg
\label{subsec:U(Gamma) parametrising}
Let $\Gamma$ be a $G$-prebrick. Define 
\[
C(\Gamma) := \langle \Gamma \rangle / \langle B(\Gamma)\rangle.
\]
The module $C(\Gamma)$ is a torus invariant $G$-constellation. A submodule $\sG$ of $C(\Gamma)$ is determined by a subset $A \subset \Gamma$, which forms a $\C$-basis of~$\sG$.
\begin{Lem}\label{Lem:combinatorial description of submodule}
Let $A$ be a subset of $\Gamma$. The following are equivalent.
\begin{enumerate}
\item The set $A$ forms a $\C$-basis of a submodule of $C(\Gamma)$.
\item If $\bm_{\rho} \in A$ and $\bbf \in \{x,y,z\}$, then $\bbf \cdot \bm_{\rho} \in \Gamma$ implies $\bbf \cdot \bm_{\rho} \in A$.
\end{enumerate}
\end{Lem}

Let $p$ be a point in $U(\Gamma)$. Then the evaluation map 
\[\ev{p}\colon S(\Gamma)\rightarrow (\C,\times),\]
is a semigroup homomorphism.

To assign a $G$-constellation $C(\Gamma)_p$ to the point $p$ of $U(\Gamma)$, first consider the $\C$-vector space with basis $\Gamma$ whose $G$-action is induced by the $G$-action on $\C[x,y,z]$. 
Endow it with the following $\C[x,y,z]$-action: 
\begin{align}\label{Eqtn:C(Gamma) at p}
\bn \ast \bm_{\rho} := \ev{p}
\left(\dfrac{\bn \cdot \bm_{\rho}}{\wtga(\bn \cdot \bm_{\rho})} \right)\wtga(\bn \cdot \bm_{\rho}),
\end{align}
for a monomial $\bn \in \mon$ and $\bm_{\rho} \in \Gamma$.
\begin{Lem}\label{Lem:C(Gamma)p}
Let $\Gamma$ be a $G$-prebrick.
\begin{enumerate}
\item For every $p \in U(\Gamma)$, $C(\Gamma)_{p}$ is a $G$-constellation.
\item For every $p \in U(\Gamma)$, $\Gamma$ is a $\C$-basis of $C(\Gamma)_p$.
\item If $p$ and $q$ are different points in $U(\Gamma)$, then $C(\Gamma)_{p} \not \iso C(\Gamma)_{q}$.
\item Let $Z \subset \T=(\C\mul)^3$ be a free $G$-orbit and $p$ the corresponding point in the torus $\Spec \C[M]$ of $U(\Gamma)$. Then $C(\Gamma)_{p} \iso \OZ$ as $G$-constellations.
\item If $U(\Gamma)$ has a torus fixed point $p$, then $C(\Gamma)_{p} \iso C(\Gamma)$.
\end{enumerate} 
\end{Lem}
\begin{proof}
From the definition of $C(\Gamma)_{p}$, the assertions (i), (ii) and (v) follow immediately. The assertion (iii) follows from the fact that points on the affine toric variety $U(\Gamma)$ are in 1-to-1 correspondence with semigroup homomorphisms from $S(\Gamma)$ to $\C$. 

It remains to show (iv). Let $Z \subset \T=(\C\mul)^3$ be a free $G$-orbit and $p$ the corresponding point in $\Spec \C[M] \subset U(\Gamma)$. There is a surjective $G$-equivariant $\C[x,y,z]$-module homomorphism 
\[
\C[x,y,z] \rightarrow C(\Gamma)_{p} \quad \text{given by $f \mapsto f * \one$}.
\]
whose kernel is equal to the ideal of $Z$.
This proves (iv).
\end{proof}

\begin{Def}
A $G$-prebrick is said to be {\em $\theta$-stable} if $C(\Gamma)$ is $\theta$-stable.
\end{Def}
\subsubsection*{\bf Deformation space $D(\Gamma)$} 
We introduce deformation theory of $C(\Gamma)$ for a $\theta$-stable $G$-prebrick~$\Gamma$. We deform $C(\Gamma)$, keeping the same vector space structure, but perturbing the structure of $\C[x,y,z]$-module. Since we fix a $\C$-basis $\Gamma$ of $C(\Gamma)$, deforming $C(\Gamma)$ involves $3r$ parameters $\{x_{\rho}, y_{\rho}, z_{\rho} \st \rho \in G\dual \}$ with
\[
\begin{cases}
x * \bm_{\rho} = x_{\rho} \wtga(x \cdot \bm_{\rho}),\\
y * \bm_{\rho} = y_{\rho} \wtga(y \cdot \bm_{\rho}),\\
z * \bm_{\rho} = z_{\rho} \wtga(z \cdot \bm_{\rho}),
\end{cases}
\]
with the following commutation relations:
\begin{equation}\label{Eqtn:the conmmutative relations}
\begin{cases}
x_{\rho}y_{\wt(x \cdot \bm_{\rho})} - y_{\rho}x_{\wt(y \cdot \bm_{\rho})},\\
x_{\rho}z_{\wt(x \cdot \bm_{\rho})} - z_{\rho}x_{\wt(z \cdot \bm_{\rho})},\\
y_{\rho}z_{\wt(y \cdot \bm_{\rho})} - z_{\rho}y_{\wt(y \cdot \bm_{\rho})}.
\end{cases}
\end{equation}
Note that $\wtga(\bm)\in \Gamma$ is the base of the same weight as $\bm$.
Fixing a basis $\Gamma$ means that we set $\bbf_{\rho}=1$ if $\wtga(\bbf\cdot\bm_{\rho})=\bbf\cdot\bm_{\rho}$ for $\bbf \in \{x,y,z\}$.
Define a subset of the $3r$ parameters
\[
\Lambda(\Gamma) := \big\{\bbf_{\rho} \st \wtga(\bbf\cdot\bm_{\rho})=\bbf\cdot\bm_{\rho}, \ \bbf_{\rho} \in \{x_{\rho}, y_{\rho}, z_{\rho}\} \big\},
\]
i.e.\ $\Lambda(\Gamma)$ is the set of parameters fixed to be 1.
Define the affine scheme 
\begin{equation}\label{Eqtn:coordinate ring of D(Gamma)}
D(\Gamma):=\Spec \big(\C[x_{\rho}, y_{\rho}, z_{\rho} \st \rho \in G\dual]\big{/}I_{\Gamma}\big)
\end{equation}
where $I_{\Gamma}=\big\langle \text{the quadrics in \eqref{Eqtn:the conmmutative relations}, $\bbf_{\rho}-1 \st \bbf_{\rho} \in \Lambda(\Gamma)$}\big\rangle$.
By King\cite{K94}, the affine scheme $D(\Gamma)$ is an open set of $\mth$ containing the point corresponding to $C(\Gamma)$. More precisely, consider an affine open set $\widetilde{U_{\Gamma}}$ in $\rg$, which is defined by $\bbf_{\rho}$ to be nonzero for all $\bbf_{\rho} \in \Lambda(\Gamma)$. Note that $\widetilde{U_{\Gamma}}$ is $\GL(\delta)$-invariant and that $\widetilde{U_{\Gamma}}$ is in the $\theta$-stable locus. Since the quotient map $\rgssth \rightarrow \mth$ is a geometric quotient for generic $\theta$, 
from GIT\cite{GIT}, it follows that $\Spec \C[\widetilde{U_{\Gamma}}]^{\GL(\delta)}$
is an affine open subset of $\mth$.
On the other hand, setting $\bbf_{\rho} \in \Lambda(\Gamma)$ to be 1 for all $\bbf_{\rho} \in \Lambda(\Gamma)$ gives a slice of the ${\GL(\delta)}$-action. Thus $D(\Gamma)$ is isomorphic to $\Spec \C[\widetilde{U_{\Gamma}}]^{\GL(\delta)}$.
\begin{Rem}
The affine open subset $D(\Gamma)$ of the moduli space $\mth$ parametrises $G$-constellations whose basis is $\Gamma$.\erem

\begin{Prop}\label{Prop:open immersion}
For generic $\theta$, let $\Gamma$ be a $\theta$-stable $G$-brick and $\yth$ the birational component of~$\mth$. Then $C(\Gamma)_{p}$ is $\theta$-stable for every $p \in U(\Gamma)$. Furthermore, there exists an open immersion
\[
\begin{array}{ccc}
U(\Gamma)=\Spec \C[S(\Gamma)]  & \hookrightarrow& \yth \subset \sM_{\theta},
\end{array}
\]
which fits into the following commutative diagram:
\[
\begin{array}{ccc}
U(\Gamma) & \hookrightarrow & \yth \\[4pt]
 \hookdownarrow && \hookdownarrow \\[4pt]
D(\Gamma) & \hookrightarrow & \mth
\end{array}
\]
where the vertical morphisms are closed embeddings.
\end{Prop}
\begin{proof}
Let us assume that the $G$-constellation $C(\Gamma)$ is $\theta$-stable.
Let $p$ be an arbitrary point in $U(\Gamma)$ and $\sG$ a submodule of $C(\Gamma)_{p}$. By the definition of $C(\Gamma)_{p}$, there is a submodule $\sG'$ of $C(\Gamma)$ whose support is the same as $\sG$. Since $C(\Gamma)$ is $\theta$-stable, $\theta(\sG)= \theta(\sG') >0$. Thus $C(\Gamma)_{p}$ is $\theta$-stable.

Since there is a $\C$-algebra epimorphism from $\C[D(\Gamma)]$ to $\C[S(\Gamma)]$ given by
\[
\bbf_{\rho} \mapsto \frac{\bbf\cdot \bm_{\rho}}{\wtga(\bbf\cdot \bm_{\rho})}
\]
for $\bbf_{\rho} \in \{x_{\rho}, y_{\rho}, z_{\rho}\}$,
it follows that $U(\Gamma)$ is a closed subscheme of $D(\Gamma)$. 

As Craw, Maclagan, and Thomas\cite{CMT} proved that the birational component $\yth$ is a unique irreducible component of $\mth$ containing the torus $T=(\C\mul)^3/G$, it follows that $Y_{\theta} \cap D(\Gamma)$ is a unique irreducible component of $D(\Gamma)$ containing the torus $T$. 

The morphism $U(\Gamma) \hookrightarrow D(\Gamma) \subset \mth$ induces an isomorphism between the torus $\Spec \C[M]$ and the torus $T$ of $\yth$ by Lemma~\ref{Lem:C(Gamma)p}~(iv).
Note that $U(\Gamma)$ is in the component $Y_{\theta} \cap D(\Gamma)$ because $U(\Gamma)$ is a closed subset of $D(\Gamma)$ containing $T$.
Since both $U(\Gamma)$ and $Y_{\theta} \cap D(\Gamma)$ are reduced and of the same dimension, $U(\Gamma)$ is equal to $Y_{\theta} \cap D(\Gamma)$. Thus there exists an open immersion from $U(\Gamma)$ to $\yth$.
\end{proof}
\subsection{$G$-bricks and the biratioanl component $\yth$} \label{subsec:G-bricks and the biratioanl component yth}
In this section, we present a 1-to-1 correspondence between the set of torus fixed points in $\yth$ and the set of $\theta$-stable $G$-bricks.

\begin{Prop}\label{Prop:1-to-1 corr between torus fixed points and G-bricks}
Let $G\subset \GL_3(\C)$ be the group of type $\frac{1}{r}(\alpha_1,\alpha_2,\alpha_3)$. For a generic parameter $\theta$, there is a 1-to-1 correspondence between the set of torus fixed points in the birational component $\yth$ of the moduli space $\mth$ and the set of $\theta$-stable $G$-bricks. 
\end{Prop}
\begin{proof}\footnote{In \cite{Thesis}, there is another proof using the language of the McKay quiver representations. }
In Section~\ref{subsec:G-bricks and local charts of yth}, we have seen that if $\Gamma$ is a $\theta$-stable $G$-brick, then $C(\Gamma)$ is a torus invariant $G$-constellation corresponding to a torus fixed point in $\yth$. 
 
Let $p\in\yth$ be a torus fixed point and $\sF$ the corresponding torus invariant $G$-constellation. For a one parameter subgroup
\[
\lambda\colon\C\mul \rightarrow T \subset \yth \quad 
\]
with $\lim_{t\rightarrow 0} \lambda(t) = p$,
$\lambda$ induces a flat family $\sV$ of $G$-constellations over $\A_{\C}^1$. Note that $\sV$ has generic support; for every nonzero $t\in\A_{\C}^1$, the $G$-constellation $\sV_t$ over $t$ is isomorphic to $\OZ$ for a free $G$-orbit $Z$ in $\T =(\C\mul)^3$. There is a set $\Gamma=\{\bm_{\rho}\in \lau\st \rho \in G\dual\}$ satisfying:
\begin{enumerate}
\item[(1)] $\Gamma$ is a $\C$-basis of $\sV_t$ for every $t\in \A_{\C}^1$.
\item[(2)] $\one \in \Gamma$.
\item[(3)] $\sV_t$ is isomorphic to $\langle \Gamma \rangle / N_t$ for a submodule $N_t$ of $\langle \Gamma \rangle$, where $\langle \Gamma \rangle$ denotes the $\C[x,y,z]$-module generated by~$\Gamma$.
\end{enumerate}

We prove that $\Gamma$ is a $G$-prebrick and that $\sF\iso C(\Gamma)$. Note that $\sF$ can be written as $\langle \Gamma \rangle / N$ for a submodule $N$. For any $\bm_{\rho}\in \Gamma$, since $\bm_{\rho}$ is a base, $\bm_{\rho}$ is not in $N$. 
Moreover if $\bn \cdot \bm_{\rho} \not \in \Gamma$ for $\bn\in\mon$, $\bm_{\rho} \in\Gamma$, then $\bn \cdot \bm_{\rho} \in  N$ because the dimension of $\hh^0(\sF)$ is $r=\lvert \Gamma \rvert$. This proves that $N=\langle B(\Gamma) \rangle$, where $B(\Gamma)$ is the Border bases in Lemma~\ref{Lem:crucial finitely gen}. From this, it follows that $\Gamma$ satisfies the conditions (i),(ii),(iii) in Definition~\ref{Def:G-prebrick}. As $\sF$ is $\theta$-stable for generic $\theta$, the connectedness condition (iv) follows.

To see that the $G$-prebrick $\Gamma$ is a $G$-brick, note that the point $p\in \yth$ corresponds to the isomorphism class of $C(\Gamma)$ so $p\in\D(\Gamma)$. Thus $p$ is in $U(\Gamma)=\yth \cap D(\Gamma)$.
\end{proof}
\begin{Cor}
Let $\Gamma$ be a $G$-prebrick. Then $C(\Gamma)$ lies in the birational component $\yth$ if and only if $\Gamma$ is a $G$-brick.
\end{Cor}
\begin{Thm}\label{Thm:Y theta with G-bricks}
Let $G\subset\GL_3(\C)$ be a finite diagonal group and $\theta$ a generic GIT parameter for $G$-constellations. Assume that $\gr$ is the set of all $\theta$-stable $G$-bricks.
\begin{enumerate}
\item The birational component $\yth$ of $\mth$ is isomorphic to the not-necessarily-normal toric variety  $\bigcup_{\Gamma \in \gr} U(\Gamma)$. 
\item The normalization of $\yth$ is isomorphic to the normal toric variety whose toric fan consists of the 3-dimensional cones $\sigma(\Gamma)$ for $\Gamma \in \gr$ and their faces.
\end{enumerate} 
\end{Thm}
\begin{proof}
Let $G$ be the group of type $\frac{1}{r}(\alpha_1,\alpha_2,\alpha_3)$. Consider the lattice $L = \Z^3 + \Z\cdot \frac{1}{r}(\alpha_1,\alpha_2,\alpha_3)$.

Let $\yth$ be the birational component of the moduli space of $\theta$-stable $G$-constellations and $\yth^{\nu}$ the normalization of $\yth$. Let $Y$ denote the not-necessarily-normal toric variety  $\bigcup_{\Gamma \in \gr} U(\Gamma)$. Define the fan $\Sigma$ in $L_{\R}$ whose maximal cones are $\sigma(\Gamma)$ for $\Gamma \in \gr$. Note that the corresponding toric variety $Y^{\nu}:=X_{\Sigma}$ is the normalization of $Y$.
%

By Proposition~\ref{Prop:open immersion}, there is an open immersion $\psi \colon Y \rightarrow \yth$.
From Proposition~\ref{Prop:1-to-1 corr between torus fixed points and G-bricks}, it follows that the image $\psi(Y)$ contains all torus fixed points of $\yth$.
The induced morphism ${\psi^{\nu}}\colon Y^{\nu}\rightarrow \yth^{\nu}$ is an open embedding of normal toric varieties with the same number of torus fixed points. Thus the morphism ${\psi^{\nu}}$ should be an isomorphism. This proves (ii). 

To show (i), suppose that $\yth\setminus\psi(Y)$ is nonempty so it contains a torus orbit $O$ of dimension $d \geq 1$. Since the normalization morphism is torus equivariant and surjective, there exists a torus orbit $O'$ in $Y^{\nu}\iso \yth^{\nu}$ of dimension $d$ which is mapped to the torus orbit $O$. At the same time, from the fact that $Y^{\nu}$ is the normalization of $Y$ and that the normalization morphism is finite, it follows that the image of $O'$ is a torus orbit of dimension $d$, so the image is $O$. Thus $O$ is in $\psi(Y)$, which is a contradiction.
\end{proof}
\begin{Cor}\label{Cor:When Y theta is normal}
With the notation as in Theorem~\ref{Thm:Y theta with G-bricks}, $\yth$ is a normal toric variety if and only if $S(\Gamma)=\sigma(\Gamma)\dual \cap M$ for all $\Gamma \in \gr$.
\end{Cor}

\section{Weighted blowups and economic resolutions}\label{Sec:Wt Blups and Econ Resolns}
Let $G \subset \GL_3(\C)$ be the finite subgroup of type $\frac{1}{r}(1,a,r-a)$ with $r$ coprime to $a$, i.e.\ 
\[
G=\langle\diag(\epsilon,\epsilon^a,\epsilon^{r-a}) \st \epsilon^r=1\rangle.
\]
The quotient $X=\C^3/G$ has terminal singularities and has no crepant resolution. However there exist a special kind of toric resolutions, which can be obtained by a sequence of weighted blowups. In this section, we review the notion of toric weighted blowups and define round down functions which are used for finding admissible $G$-bricks.
\subsection{Terminal quotient singularities in dimension 3}\label{subsec:term quot sings in dim 3}
In this section, we collect various facts from birational geometry. Most of these are taken from \cite{R87}.

\begin{Def}
Let $X$ be a normal quasiprojective variety and $\KX$ the canonical divisor on $X$. We say that $X$ has {\em terminal singularities} (resp.\ {\em canonical singularities}) if it satisfies the following conditions:
\begin{enumerate}
\item there is a positive integer $r$ such that $r\KX$ is a Cartier divisor.
\item if $\varphi \colon Y \rightarrow X$ is a resolution with $E_i$ prime exceptional divisors such that
\[
r\KY = \varphi\pull (r\KX) + r \sum a_i E_i,
\]
then $a_i >0\text{ (resp.\ $\geq 0$)}$ for all $i$.
\end{enumerate}
\end{Def}
In the definition above, $a_i$ is called the {\em discrepancy} of $E_i$. A {\em crepant resolution} $\varphi$ of $X$ is a resolution with all discrepancies zero.
\subsubsection*{\bf Birational geometry of toric varieties}
Let $L$ be a lattice of rank $n$ and $M$ the dual lattice of $L$.
Let $\sigma$ be a cone in $L\otimes_{\Z}\R$. Fix a primitive element $v \in L \cap \sigma$. The {\em barycentric subdivision} of $\sigma$ at $v$ is the minimal fan containing all cones $\Cone(\tau,v)$ where $\tau$ varies over all subcones of $\sigma$ with $v \not \in \tau$.
\begin{Prop}[see e.g.\ \cite{R87}] \label{Prop:calculation of discrepancy}Let $\Sigma$ be the barycentric subdivision of an $n$-dimensional cone $\sigma$ at $v$. Let $X:=U_{\sigma}$ be the affine toric variety corresponding to $\sigma$ and $Y$ the toric variety corresponding to the fan $\Sigma$.
\begin{enumerate}
\item The barycentric subdivision induces a projective toric morphism
\[
\varphi\colon Y \rightarrow X.
\]
\item The set of 1-dimensional cones of $\Sigma$ consists of 1-dimensional cones of $\sigma$ and $\Cone(v)$.
\item The torus invariant prime divisor $D_v$ corresponding to the 1-dimensional cone $\Cone(v)$ is a $\Q$-Cartier divisor on $Y$.
\end{enumerate}
Furthermore if $v$ is an interior lattice point in $\sigma$, then
\[
\KY = \varphi\pull(\KX) + (\langle x_1x_2\cdots x_n,v \rangle -1)D_v, 
\]
i.e.\ the discrepancy of the exceptional divisor $D_v$ is $\langle x_1x_2\ldots x_n,v \rangle -1$.
\end{Prop}

\begin{Eg}\label{Eg:discrepancy of 1/r(1,a,r-a)}
Define the lattice $L=\Z^3 +\Z\cdot \tfrac{1}{r}(1,a,r-a)$ with $r$ coprime to $a$ and $M=\Homz(L,\Z)$ the dual lattice. 
Let $\{e_1,e_2,e_3\}$ be the standard basis of $\Z^3$ and $\sigma_{+}$ the cone generated by $e_1,e_2,e_3$. 
As in Section~\ref{subsec:G-bricks and local charts of yth}, the toric variety $X:=U_{\sigma_+}$ is the quotient variety $\C^3/G$ where $G$ is the group of type $\tfrac{1}{r}(1,a,r-a)$.

Set $v_i:=\frac{1}{r}(i,\overline{ai},\overline{r-ai}) \in L$ for each $1 \leq i < r-1$ where $\bar{\quad}$ denotes the residue modulo $r$. 
Let $E_i$ be the torus invariant prime divisor corresponding to $v_i$. From Proposition~\ref{Prop:calculation of discrepancy}, the discrepancy of $E_i$ is 
\[
\frac{i}{r}+\frac{\overline{ai}}{r}+\frac{r-\overline{ai}}{r}-1=\frac{i}{r}>0.
\]
This shows that $X$ has only terminal singularities.\eeg
We have seen that the quotient singularity $X= \C^3/G$ has terminal singularities if $G$ is the group of type $\frac{1}{r}(1,a,r-a)$ with $r$ coprime to~$a$. Conversely, these groups are essentially all the cases, by the following.
\begin{Thm}[Morrison and Stevens\cite{MS}]
A 3-fold cyclic quotient singularity $X= \C^3/G$ has terminal singularities if and only if the group $G\subset \GL_3(\C)$ is of type $\frac{1}{r}(1,a,r-a)$ with $r$ coprime to $a$.
\end{Thm}
\subsection{Weighted blowups and round down functions}\label{subsec:Round Down}
Define the lattice $L=\Z^3 +\Z\cdot \tfrac{1}{r}(1,a,r-a)$. Set $\zzz = \Z^3 \subset L$. Consider the two dual lattices $M=\Homz(L,\Z)$, $\lau=\Homz(\zzz,\Z)$. Note that a (Laurent) monomial $\bm \in \lau$ is $G$-invariant if and only if $\bm$ is in $M$.
Let $\{e_1,e_2,e_3\}$ be the standard basis of $\Z^3$ and $\sigma_{+}$ the cone generated by $e_1,e_2,e_3$. Then $\Spec \C[\sigma_{+}\dual \cap M]$ is the quotient variety $X=\C^3/G$. Set $v=\tfrac{1}{r}(1,a,r-a) \in L$, which corresponds to the exceptional divisor of the smallest discrepancy (see Example~\ref{Eg:discrepancy of 1/r(1,a,r-a)}). 
Define three cones
\begin{align*}
\sigma_{1}=\Cone(v,e_2,e_3), \quad \sigma_{2}=\Cone(e_1,v,e_3), \quad\sigma_{3}=\Cone(e_1,e_2,v).
\end{align*}
Define $\Sigma$ to be the fan consisting of the three cones $\sigma_{1},\sigma_{2},\sigma_{3}$ and their faces.
The fan $\Sigma$ is the barycentric subdivision of $\sigma_{+}$ at $v$.
Let $Y_1$ be the toric variety corresponding to the fan $\Sigma$ together with the lattice $L$. 
The induced toric morphism $\varphi \colon Y_1 \rightarrow X$ is called {\em the weighted blowup} of $X$ with weight $(1,a,r-a)$.
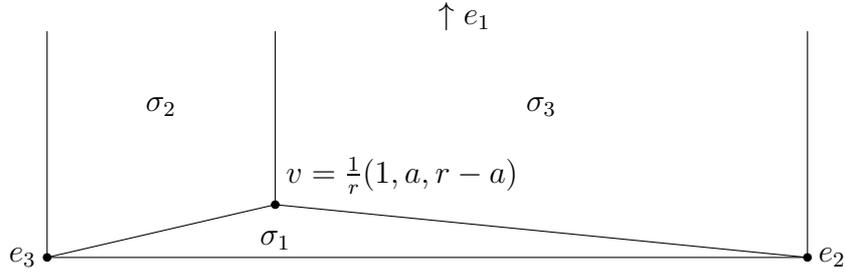
\begin{figure}[h]
\begin{center}
\begin{tikzpicture}
\coordinate [label=left:$e_3$] (e3) at (0,0);
\coordinate [label=right:$e_2$] (e2) at (10,0);
\coordinate (v1) at (3,0.7);
\draw[fill] (v1) circle [radius=0.05];
\draw[fill] (e3) circle [radius=0.05];
\draw[fill] (e2) circle [radius=0.05];
\node [above right] at (v1) {$v=\frac{1}{r}(1,a,r-a)$};
\draw (e3) -- (e2);
\draw (e3) -- (0,3);
\draw (e2) -- (10,3);
\draw (e3) -- (v1);
\draw (e2) -- (v1);
\draw (v1) -- (3,3);
\node [below] at (3,0.5) {$\sigma_{1}$};
\node  at (1.5,2) {$\sigma_{2}$};
\node  at (6.5,2) {$\sigma_{3}$};
\node [right] at (5,3.2) {$\uparrow e_1$};
\end{tikzpicture}
\end{center}
\caption{Weighted blowup of weight $(1,a,r-a)$}\label{Fig:wt blowup for 1/r(1,a,r-a)}
\end{figure}

Let us consider the sublattice $L_{2}$ of $L$ generated by $e_1,v,e_3$. Define $M_{2} := \Homz(L_{2},\Z)$ with the dual basis
\[
\xi_2:=x y^{-\frac{1}{a}}, \quad \eta_2:=y^{\frac{r}{a}}, \quad \zeta_2:=y^{\frac{a-r}{a}}z.
\] 
The lattice inclusion $L_{2} \hookrightarrow L$ induces a toric morphism 
\begin{equation*}
\varphi \colon \Spec \C[\sigma_{2}\dual \cap M_{2}] \rightarrow U_{2}:=\Spec \C[\sigma_{2}\dual \cap M].
\end{equation*}
Since $\C[\sigma_{2}\dual \cap M_{2}] \iso \C[\xi_2, \eta_2, \zeta_2]$ and the group $G_{2}:=L/L_{2}$ is of type $\lgr$ with eigencoordinates $\xi_2, \eta_2,\zeta_2$, the open subset $U_{2}$ has a quotient singularity of type $\lgr$. 

Similarly, consider the sublattice $L_3$ of $L$ generated by $e_1,e_2,v$. Let us define the lattice $M_{3} := \Homz(L_{3},\Z)$ with basis
\[
\xi_3:=x z^{-\frac{1}{r-a}}, \quad \eta_3:=yz^{\frac{-a}{r-a}}, \quad \zeta_3:=z^{\frac{r}{r-a}}.
\]
The open set $U_3=\Spec \C[\xi_3, \eta_3, \zeta_3]$ has a quotient singularity of type $\rgr$ with eigencoordinates $\xi_3, \eta_3,\zeta_3$. Set $G_{3}:=L/L_{3}$.

Lastly, consider the sublattice $L_1$ of $L$ generated by $v,e_2,e_3$. Let us define $M_{1} := \Homz(L_{1},\Z)$ with the dual basis
\[
\xi_1:=x^{\frac{1}{r}}, \quad \eta_1:=x^{-\frac{a}{r}}y, \quad \zeta_1:=x^{-\frac{r-a}{r}}z.
\]
Since $\{v,e_2,e_3\}$ forms a $\Z$-basis of $L$, i.e.\ $G_1=L/L_{1}$ is the trivial group, the open set $U_1=\Spec \C[\xi_1, \eta_1, \zeta_1]$ is smooth.
\begin{Eg} \label{Eg:wt blup of 1/7(1,3,4)} Let $G$ be the group of type $\frac{1}{7}(1,3,4)$ as in Example~\ref{Eg:G graph 1/7(1,3,4) }. The toric fan of the weighted blowup with weight $(1,3,4)$ is shown in Figure~\ref{Fig:Wt blowup of weight (1,3,4)}.

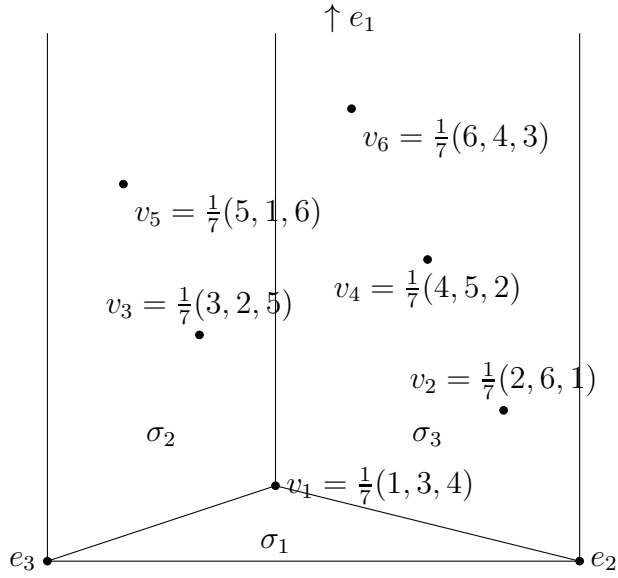
\begin{figure}[h]
\begin{center}
\begin{tikzpicture}
\coordinate [label=left:$e_3$] (e3) at (0,0);
\coordinate [label=right:$e_2$] (e2) at (7,0);
\coordinate (v1) at (3,1);
\coordinate (v2) at (6,2);
\coordinate (v3) at (2,3);
\coordinate (v4) at (5,4);
\coordinate (v5) at (1,5);
\coordinate (v6) at (4,6);

\foreach \t in {1,2,3,4,5,6}
\draw[fill] (v\t) circle [radius=0.05];
\foreach \t in {2,3}
\draw[fill] (e\t) circle [radius=0.05];

\node [right] at (v1) {$v_1=\frac{1}{7}(1,3,4)$};
\node [above] at (v2) {$v_2=\frac{1}{7}(2,6,1)$};
\node [above] at (v3) {$v_3=\frac{1}{7}(3,2,5)$};
\node [below] at (v4) {$v_4=\frac{1}{7}(4,5,2)$};
\node [below right] at (v5) {$v_5=\frac{1}{7}(5,1,6)$};
\node [below right] at (v6) {$v_6=\frac{1}{7}(6,4,3)$};

\draw (e3) -- (e2);
\draw (e3) -- (0,7);
\draw (e2) -- (7,7);
\draw (e3) -- (v1);
\draw (e2) -- (v1);
\draw (v1) -- (3,7);
\node [below] at (3,0.5) {$\sigma_{1}$};
\node  at (1.5,1.65) {$\sigma_{2}$};
\node  at (5,1.65) {$\sigma_{3}$};
\node [right] at (3.5,7.2) {$\uparrow e_1$};
\end{tikzpicture}
\end{center}
\caption{Weighted blowup of weight $(1,3,4)$}\label{Fig:Wt blowup of weight (1,3,4)}
\end{figure}

The affine toric variety corresponding to the cone $\sigma_2$ on the left side of $v=\frac{1}{7}(1,3,4)$ has a quotient singularity of type $\frac{1}{3}(1,2,1)$ with eigencoordinates $x y^{-\frac{1}{3}}, y^{\frac{7}{3}}, y^{-\frac{4}{3}}z$. The affine toric variety corresponding to the cone $\sigma_3$ on the right side of $v$ has a singularity of type $\frac{1}{4}(1,3,1)$ with eigencoordinates $x z^{-\frac{1}{4}}, yz^{-\frac{3}{4}}, z^{\frac{7}{4}}$.
On the other hand, the affine toric variety corresponding to the cone $\sigma_1=\Cone(e_2,e_3,v)$ is smooth as $e_2,e_3,v$ form a $\Z$-basis of $L$.\eeg
\begin{Def}[Round down functions]
With the notation above, the {\em left round down function} $\phi_{2} \colon \lau \to  M_{2}$ of the weighted blowup with weight $(1,a,r-a)$ is defined by
\[
\phi_{2}(x^{m_1}y^{m_2}z^{m_3})=\xi_2^{m_1}\eta_2^{\lf \frac{1}{r}m_1 + \frac{a}{r} m_2 + \frac{r-a}{r}m_3 \rf}\zeta_2^{m_3}.
\]
where $\lfloor \ \rfloor $ is the floor function. In a similar manner, the {\em right round down function} $\phi_{3} \colon \lau \to  M_{3}$ of the weighted blowup with weight $(1,a,r-a)$ is defined by
\[
\phi_{3}(x^{m_1}y^{m_2}z^{m_3})=\xi_3^{m_1}\eta_3^{m_2}\zeta_3^{\lf \frac{1}{r}m_1 + \frac{a}{r} m_2 + \frac{r-a}{r}m_3 \rf},
\]
and the {\em central round down function} $\phi_{1} \colon \lau \to  M_{1}$ of the weighted blowup with weight $(1,a,r-a)$ by
\[
\phi_{1}(x^{m_1}y^{m_2}z^{m_3})=\xi_1^{\lf \frac{1}{r}m_1 + \frac{a}{r} m_2 + \frac{r-a}{r}m_3 \rf}\eta_1^{m_2}\zeta_1^{m_3}.
\]
\end{Def}
\begin{Lem}\label{Lem:phi(m+n)=phi(m)+(n)/phi_k induces a surjective map G dual -> G_k dual}
For each $k=1,2,3$, let $\phi_k$ be the round down function of the weighted blowup with weight $(1,a,r-a)$. For a monomial $\bm \in \lau$ and a $G$-invariant monomial $\bn \in M$, 
\begin{equation*}
\phi_{k}(\bm \cdot \bn)= \phi_{k}(\bm)\cdot\bn.
\end{equation*}
Thus the weight of $\phi_{k}(\bm \cdot \bn)$ and the weight of $\phi_{k}(\bm)$ are the same in terms of the $G_k$-action. Therefore $\phi_k$ induces a surjective map
\[
\phi_k \colon G\dual \rightarrow G_k\dual,\quad \rho \mapsto \phi_k(\rho),
\]
where $\phi_k(\rho)$ is the weight of $\phi_k(\bm)$ for a monomial $\bm \in \lau$ of weight~$\rho$.
\end{Lem}
\begin{proof}
Since $M_{k}$ contains $M$ as the lattice of $G_k$-invariant monomials, $\bn$ is in $M_{k}$. By definition, the assertions follow.
\end{proof}
\begin{Rem}
Davis, Logvinenko, and Reid\cite{DLR} introduced a related construction in a more general setting.
\erem
\begin{Lem}\label{Lem:localizations,lacing}
For each $k=1,2,3$, let $\phi_k$ be the round down function of the weighted blowup with weight $(1,a,r-a)$. Let $\bm \in \lau$ be a Laurent monomial of weight $j$. 
\begin{enumerate}
\item If $0\leq j<r-a$, then $\phi_{2}(y\cdot \bm)=\phi_{2}(\bm)$.
\item If $0\leq j<a$, then $\phi_{3}(z\cdot \bm)=\phi_{3}(\bm)$.
\item If $0\leq j<r-1$, then $\phi_{1}(x\cdot \bm)=\phi_{1}(\bm)$.
\item If $\phi_k(\bm)=\phi_k(\bm')$, then $\bm=\bn\cdot \bm'$ or $\bm'=\bn\cdot \bm$ for some $\bn\in\mon$.
\end{enumerate}
\end{Lem}
\begin{proof}
Let $\bm=x^{m_1}y^{m_2}z^{m_3}$ be a Laurent monomial of weight $j$. To prove (i), assume that $0\leq j<r-a$. This means that
\[
0 \leq \frac{1}{r}m_1 + \frac{a}{r} m_2 + \frac{r-a}{r}m_3 - \lf \frac{1}{r}m_1 + \frac{a}{r} m_2 + \frac{r-a}{r}m_3 \rf <\frac{r-a}{r}.
\]
Thus $\phi_{2}(y\cdot \bm)=\phi_{2}(x^{m_1}y^{m_2+1}z^{m_3})=\phi_{2}(x^{m_1}y^{m_2}z^{m_3})=\phi_{2}(\bm)$.
The assertions (ii) and (iii) can be proved similarly. The definition of $\phi_k$ implies (iv).
\end{proof}
\begin{Lem}\label{Lem:Connected Gamma}
For each $k=1,2,3$, let $\phi_k$ be the round down function of the weighted blowup with weight $(1,a,r-a)$. Let $\bk$ be a lattice point in the monomial lattice $M_k$ and $\bg$ a monomial of degree 1 in $M_k$.
There exist a monomial $\bbf \in \{x,y,z\}$ of degree 1 and $\bm \in \lau$ such that
\[
\phi_k(\bbf \cdot \bm)=\bg\cdot\bk
\] 
satisfying $\phi_k(\bm)= \bk$.
\end{Lem}
\begin{proof}
Here we prove the assertion for the left round down function. Let $\xi,\eta,\zeta$ denote the eigencoordinates for the $G_2$-action. 
Let $\bk$ be a monomial in $M_2$ and $\bg \in \{\xi,\eta,\zeta\}$. 

Consider the case where $\bg=\zeta$. Since $\phi_2$ is surjective, there exists $\bm=x^{m_1}y^{m_2}z^{m_3} \in \lau$ such that $\phi_2(\bm)=\bk$.
If $\zeta \cdot\bk=\phi_2(z \cdot \bm)$, then we are done. 

Suppose $\zeta \cdot\bk \neq \phi_{2}(z \cdot\bm)$. This means that
\begin{equation*}
\frac{1}{r}m_1 + \frac{a}{r}m_2 + \frac{r-a}{r}m_3 + \frac{r-a}{r} \geq \left\lf\frac{1}{r}m_1 + \frac{a}{r}m_2 + \frac{r-a}{r}m_3\right\rf +1.
\end{equation*}
There is a positive integer $l_0$\footnote{This integer $l_0$ is the largest integer satisfying 
\begin{equation*}
\frac{1}{r}m_1 + \frac{a}{r}m_2 + \frac{r-a}{r}m_3 - \frac{a}{r}l \geq \left\lf\frac{1}{r}m_1 + \frac{a}{r}m_2 + \frac{r-a}{r}m_3\right\rf.
\end{equation*}
}
such that $\phi_{2}(\frac{\bm}{y^l})=\bk$ for all $0 \leq l \leq l_0$ with $\phi_{2}\Big(\frac{\bm}{y^{l_0+1}} \Big)\neq \bk$. Since $\phi_{2}(z \cdot \frac{\bm}{y^{l_0}})=\zeta \cdot \bk$, the assertion follows.

For the other cases, we can prove the assertion similarly.
\end{proof}

\subsection{Economic resolutions}\label{subsec:Econ resoln}
By the fact that the quotient variety $X=\C^3/G$ has terminal singularities, $X$ does not admit crepant resolutions. However $X$ has a certain toric resolution introduced by Danilov\cite{D82} (see also~\cite{R87}).
\begin{Def}
Let $G\subset \GL_3(\C)$ be the group of type $\frac{1}{r}(1,a,r-a)$.
For each $1 \leq i < r$, let $v_i:=\frac{1}{r}(i,\overline{ai},\overline{r-ai}) \in L$ where $\bar{\quad}$ denotes the residue modulo $r$. The {\em economic resolution} of $\C^3/G$ is the toric variety obtained by the consecutive weighted blowups at $v_1,v_2,\ldots,v_{r-1}$ from $\C^3/G$.
\end{Def}
\begin{Prop}[see \cite{R87}]
Let $\varphi \colon Y \rightarrow X=\C^3/G$ be the economic resolution of $\C^3/G$. For each $1 \leq i < r$, let $E_i$ denote the exceptional divisor of $\varphi$ corresponding to the lattice point $v_i$.
\begin{enumerate}
\item The toric variety $Y$ is smooth and projective over $X$.
\item The morphism $\varphi$ satisfies
\[\KY=\varphi\pull(\KX)+\sum_{1 \leq i < r}\frac{i}{r}E_i.\] In particular, each discrepancy is $0< \frac{i}{r} <1$.
\end{enumerate}
\end{Prop}
From the fan of $Y$, we can see that $Y$ can be covered by three open sets $U_2$, $U_3$ and $U_1$, which are the unions of the affine toric varieties corresponding to the cones on the left side of, the right side of, and below the vector $v=\frac{1}{r}(1,a,r-a)$, respectively. Note that $U_2$ and $U_3$ are isomorphic to the economic resolutions for the singularity of type $\lgr$ and of type $\rgr$, respectively. 
\begin{Rem}
Let $\Sigma$ be the toric fan of the economic resolution $Y$. Note that the number of 3-dimensional cones in $\Sigma$ is $2r-1$ and that the number of 3-dimensional cones containing $e_1$ is $r$.\erem
\begin{Eg} \label{Eg:econ. resolns of 1/7(1,3,4)} Let $G$ be the group of type $\frac{1}{7}(1,3,4)$ as in Example~\ref{Eg:G graph 1/7(1,3,4) }. The fan of the economic resolution of the quotient variety $\C^3/G$ is shown in Figure~\ref{Fig:Fan of the economic resolution for 1/7(1,3,4)}.
\begin{figure}[h]
\begin{center}
\begin{tikzpicture}
\coordinate [label=left:$e_3$] (e3) at (0,0);
\coordinate [label=right:$e_2$] (e2) at (7,0);
\foreach \x in {1,2,3,4,5,6}
    \draw (\x ,2pt) -- (\x ,-2pt);
\foreach \y in {1,2,3,4,5,6}
    \draw (2pt,\y) -- (-2pt,\y);
\coordinate (v1) at (3,1);
\coordinate (v2) at (6,2);
\coordinate (v3) at (2,3);
\coordinate (v4) at (5,4);
\coordinate (v5) at (1,5);
\coordinate (v6) at (4,6);

\foreach \t in {1,2,3,4,5,6}
\draw[fill] (v\t) circle [radius=0.05];
\foreach \t in {2,3}
\draw[fill] (e\t) circle [radius=0.05];
\node [below] at (3,0.8) {$v=\frac{1}{7}(1,3,4)$};
\foreach \t in {2,3,4,5,6}
\node [right] at (v\t) {$v_{\t}$};

\draw (e3) -- (e2);
\draw (e3) -- (0,7);
\draw (e2) -- (7,7);

\draw (v1) -- (3,7);
\draw (v2) -- (6,7);
\draw (v3) -- (2,7);
\draw (v4) -- (5,7);
\draw (v5) -- (1,7);
\draw (v6) -- (4,7);

\draw (v1) -- (v5);
\draw (e2) -- (v6);

\foreach \t in {1,3,5}
\draw (e3) -- (v\t);
\foreach \t in {1,6}
\draw (e2) -- (v\t);
\foreach \t in {2,4,6}
\draw (v1) -- (v\t);

\node at (2.5,6.5) {$\sigma_{1}$};
\node at (3.5,6.5) {$\sigma_{2}$};

\node [right] at (3.5,7.2) {$\uparrow e_1$};

\end{tikzpicture}
\end{center}
\caption{Fan of the economic resolution for $\frac{1}{7}(1,3,4)$}\label{Fig:Fan of the economic resolution for 1/7(1,3,4)}
\end{figure}
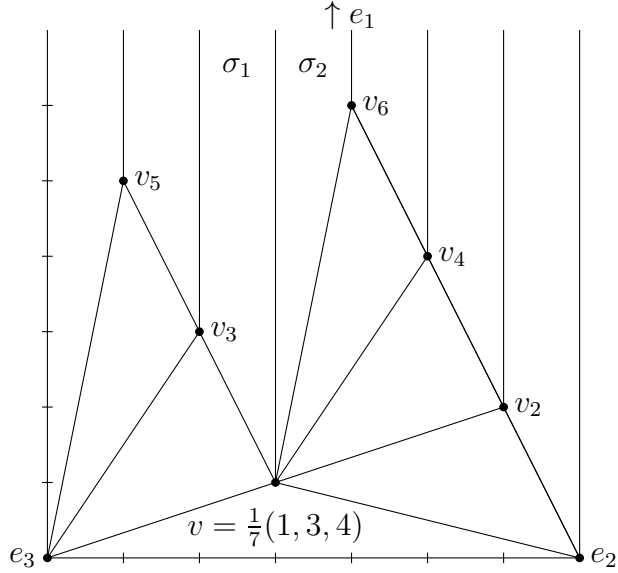

The toric variety corresponding to the fan consisting of the cones on the left side of $v=\frac{1}{7}(1,3,4)$ is the economic resolution of the quotient $\frac{1}{3}(1,2,1)$. On the other hand, the toric variety corresponding to the fan consisting of the cones on the right side of $v$ is the economic resolution of the quotient $\frac{1}{4}(1,3,1)$.\eeg
\section{Moduli interpretations of economic resolutions}\label{Sec:Main Theorem}
This section contains our main theorem. First, we explain how to find a  set $\gr(r,a)$ of $G$-bricks using the round down functions and a recursion process. In Section~\ref{subsec:Our stability parameters}, we show that there exists a stability parameter $\theta$ such that every $G$-brick in $\gr(r,a)$ is $\theta$-stable.
\subsection{$G$-bricks and stability parameters for $\frac{1}{r}(1,r-1,1)$}\label{subsec:G-bricks for 1/r(1,r-1,1)}
Let $G$ be the group of $\frac{1}{r}(1,r-1,1)$ type, i.e.\ $a=1 \text{ or } r-1$. In this case, K\k{e}dzierski\cite{Ked04} proved that $\GHilb{3}$ is isomorphic to the economic resolution of $\C^3/G$.
\begin{Thm}[K\k{e}dzierski\cite{Ked04}]\label{Thm:Ked. a=1 smooth}
Let $G \subset \GL_3(\C)$ be the finite group of type $\frac{1}{r}(1,a,r-a)$ with $a=1 \text{ or } r-1$. Then
$\GHilb{3}$ is isomorphic to the economic resolution of the quotient variety $\C^3/G$. 
\end{Thm}
For each $1\leq i < r$, set $v_i=\frac{1}{r}(i,r-i,i)$. Set $v_0=e_2$ and $v_{r}=e_3$. The toric fan corresponding to 
$\GHilb{3}$ consists of the following $2r-1$ maximal cones and their faces:
\begin{align*}
\sigma_i=\Cone (e_1, v_{i-1}, v_i) \qquad &\text{ for $1\leq i \leq r$,}\\
 \sigma_{r+i}=\Cone (e_3, v_{i-1}, v_i) \qquad &\text{ for $1\leq i \leq r-1$.}
\end{align*}
Each 3-dimensional cone has a corresponding (Nakamura's) $G$-graph:
\begin{equation}
\begin{array}{rl}\label{Eqtn:G-bricks for a=1}
\Gamma_i =\{1,y,y^2,\ldots,y^{i-1},z,z^2,\ldots,z^{r-i} \} \quad &\text{ for $1\leq i \leq r$,}\\[4pt]
\Gamma_{r+i} =\{1,y,y^2,\ldots,y^{i-1},x,x^2,\ldots,x^{r-i} \} \quad &\text{ for $1\leq i \leq r-1$,}
\end{array}
\end{equation}
with $S(\Gamma_j)=\sigma_j\dual\cap M$. As the cone $\sigma_j$ is 3-dimensional, the $G$-prebrick $\Gamma_j$ is a $G$-brick. Furthermore, $U(\Gamma_j)=D(\Gamma_j)\iso \C^3$.

By Ito-Nakajima\cite{IN}, all $G$-bricks in \eqref{Eqtn:G-bricks for a=1} are $\theta$-stable for any $\theta\in \Theta_+$ where 
\begin{equation}\label{Eqtn:Stab for G-Hilb}
\Theta_+ := \left\{\theta \in \Theta \st \theta\left(\rho\right) >0 \text{ for } \rho \neq \rho_0 \right\}.
\end{equation}
\begin{Eg}\label{Eq:G-iraffes 1/3(1,2,1)}
Let $G$ be the finite group of type $\frac{1}{3}(1,2,1)$ with eigencoordinates $\xi, \eta, \zeta$. Set $v_1=\frac{1}{3}(1,2,1)$ and $v_2=\frac{1}{3}(2,1,2)$. Recall that the economic resolution $Y$ of $X=\C^3/G$ can be obtained by the sequence of the weighted blowups:
\[
Y \stackrel{\varphi_2}{\longrightarrow} Y_1 \stackrel{\varphi_1}{\longrightarrow} X ,
\]
where $\varphi_1$ is the weighted blowup with weight $(1,2,1)$ and $\varphi_2$ is the toric morphism induced by the weighted blowup with weight $(2,1,2)$. 
The fan corresponding to 
$Y$ consists of the following five 3-dimensional cones and their faces:
\[
\begin{array}{lll}
\sigma_1=\Cone (e_1, e_2, v_1), \quad & \sigma_2=\Cone (e_1, v_1, v_2),\quad & \sigma_3=\Cone (e_1, v_2, e_3),\\[2pt]
\sigma_4=\Cone (e_3, e_2, v_1),\quad & \sigma_5=\Cone (e_3, v_1, v_2). &
\end{array}
\]
The following
\[
\begin{array}{lll}
\Gamma_1=\{1,\zeta,\zeta^2 \},  \quad &
\Gamma_2=\{1,\eta,\zeta \}, \quad  
\Gamma_3=\{1,\eta,\eta^2 \}, \\[3pt]
\Gamma_4=\{1,\xi,\xi^2 \}, \quad &
\Gamma_5=\{1,\xi,\eta \}.
\end{array}
\]
are their corresponding $G$-bricks.
\eeg
\subsection{$G$-bricks for $\frac{1}{r}(1,a,r-a)$}\label{subsec:G-bricks for 1/r(1,a,r-a)}
In this section, we assign a $G$-brick $\Gamma_{\sigma}$ with $S(\Gamma_{\sigma})=\sigma\dual\cap M$ to each maximal cone $\sigma$ in the fan of the economic resolution $Y$.

Let $G$ be the group of type $\frac{1}{r}(1,a,r-a)$ with $r$ coprime to~$a$, $X$ the quotient $\C^3/G$, and $\varphi \colon Y \rightarrow X$ the economic resolution of $X$. Then $Y$ can be covered by $U_2$, $U_3$ and $U_1$, which are the unions of the affine toric varieties corresponding to the cones on the left side of,  the right side of, and below the lattice point $v = \frac{1}{r}(1,a,r-a)$, respectively.

\begin{Prop} \label{Prop:From Gamma' to Gamma:S(Gamma)=S(Gamma')}
For $k=1,2,3$, let $\Gamma'$ be a $G_k$-brick. Define
\[
\Gamma:=\left\{\bm \in \lau \st \phi_{k}(\bm) \in \Gamma' \right\}.
\]
The set $\Gamma$ is a $G$-brick with $S(\Gamma)=S(\Gamma')$. 
\end{Prop}
\begin{proof}
Since $\phi_{k}(\one)=\one \in \Gamma'$, we have $\one \in \Gamma$. To show that $\Gamma$ satisfies (ii) in Definition~\ref{Def:G-prebrick}, we need to show that there exists a unique monomial of weight $\rho$ in $\Gamma$ for each $\rho \in G\dual$.
Fix a positive integer $i$ such that the weight of $x^i$ is $\rho$. Consider the monomial $\phi_{k}(x^i)$ in $M_{k}$ and its weight $\chi\in G_{k}\dual$. Since $\Gamma'$ is a $G_{k}$-brick, there exists a unique element $\bk_{\chi}$ of weight $\chi$. Since the $G_k$-invariant monomial $\frac{\bk_{\chi}}{\phi_{k}(x^i)}$ is in the lattice $M$, it follows from Lemma~\ref{Lem:phi(m+n)=phi(m)+(n)/phi_k induces a surjective map G dual -> G_k dual} that
\[
\phi_{k}\colon x^i \cdot \Big(\frac{\bk_{\chi}}{\phi_{2}(x^i)} \Big) \mapsto \bk_{\chi},
\]
i.e.\ $x^i \cdot \left(\frac{\bk_{\chi}}{\phi_{2}(x^i)} \right)$ is in $\Gamma$. To show uniqueness, assume that two monomials $\bm,\bm'$ of the same weight are mapped into $\Gamma'$. As $\phi_{k}(\bm)$ and $\phi_{k}(\bm')$ are of the same weight, we have $\phi_{k}(\bm)=\phi_{k}(\bm')\in \Gamma'$. From Lemma~\ref{Lem:phi(m+n)=phi(m)+(n)/phi_k induces a surjective map G dual -> G_k dual},
\begin{equation*}
\phi_{k}(\bm)=\phi_{k}\left(\bm' \cdot\frac{\bm}{\bm'}\right)=\phi_{k}(\bm')\cdot\frac{\bm}{\bm'},
\end{equation*}
and hence $\bm=\bm'$.
From Lemma~\ref{Lem:Connected Gamma}, it follows that $\Gamma$ is connected as $\Gamma'$ is connected.

For (iii) in Definition~\ref{Def:G-prebrick}, assume that $\bn' \cdot \bn \cdot \bm_{\rho} \in \Gamma$ for $\bm_{\rho}\in \Gamma$ and $\bn, \bn' \in \mon$. We need to show $\phi_{k}(\bn \cdot \bm_{\rho})\in\Gamma'$. From
\[
\phi_{k}(\bn' \cdot \bn \cdot \bm_{\rho})=
\frac{\phi_{k}(\bn' \cdot \bn \cdot \bm_{\rho})}{\phi_{k}(\bn \cdot \bm_{\rho})} 
\cdot 
\frac{\phi_{k}(\bn \cdot \bm_{\rho})}{\phi_{k}(\bm_{\rho})}
\cdot 
{\phi_{k}(\bm_{\rho})}
\in \Gamma',
\]
it follows that $\phi_{k}(\bn \cdot \bm_{\rho})=\frac{\phi_{k}(\bn \cdot \bm_{\rho})}{\phi_{k}(\bm_{\rho})}
\cdot 
{\phi_{k}(\bm_{\rho})}$ is in $\Gamma'$ because $\Gamma'$ is a $G_k$-prebrick. This proves that $\Gamma$ is a $G$-prebrick.

It remains to prove that $S(\Gamma)=S(\Gamma')$.
Note that $S(\Gamma)$ is generated by $\frac{\bn \cdot \bm_{\rho}}{\wtga(\bn \cdot \bm_{\rho})}$ for $\bn \in \mon$ and $\bm_{\rho} \in \Gamma$. Let $\bn$ be a genuine monomial in $\mon$ and $\bm_{\rho}$ an element in $\Gamma$. Let $\bk_{\chi}$ denote $\phi_{k}(\bm_{\rho}) \in \Gamma'$. Define $\bk$ to be $\frac{\phi_{k}(\bn\cdot \bm_{\rho})}{\phi_{k}(\bm_{\rho})}$. From the definition of the round down functions, we know that $\bk$ is a genuine monomial in $\xi, \eta, \zeta$. Since $\frac{\bn \cdot \bm_{\rho}}{\wtga(\bn \cdot \bm_{\rho})}$ is $G$-invariant, it follows that
\[
\frac{\bn \cdot \bm_{\rho}}{\wtga(\bn \cdot \bm_{\rho})} 
= \frac{\phi_{k}(\bn \cdot \bm_{\rho})}{\phi_{k}\big(\wtga(\bn \cdot \bm_{\rho})\big)}
= \frac{\frac{\phi_{k}(\bn \cdot \bm_{\rho})}{\phi_{k}( \bm_{\rho})} \cdot \phi_{k}(\bm_{\rho})}{\phi_{k}\big(\wtga(\bn \cdot \bm_{\rho})\big)}
=\frac{\bk \cdot \bk_{\chi}}{\wtgp(\bk \cdot \bk_{\chi})}
\]
from Lemma~\ref{Lem:phi(m+n)=phi(m)+(n)/phi_k induces a surjective map G dual -> G_k dual}.
This proves $S(\Gamma) \subset S(\Gamma')$.

For the opposite inclusion, by Lemma~\ref{Lem:crucial finitely gen}, it suffices to show that $\frac{\bg \cdot \bk_{\chi}}{\wtgp(\bg \cdot \bk_{\chi})}$ is in $S(\Gamma)$ for all $\bg\in\{\xi,\eta,\zeta\}$ and $\bk_{\chi}\in\Gamma'$.
By Lemma~\ref{Lem:Connected Gamma} there are $\bn \in \mon$, $\bm_{\rho} \in \Gamma$ such that $\phi_{k}(\bn \cdot \bm_{\rho})=\bg\cdot\bk_{\chi}$. Lemma~\ref{Lem:wt phi2(m mrho) = phi2 (wt (m mrho)} implies that $\wtgp(\bg \cdot \bk_{\chi})=\phi_{k}\big(\wtga(\bn \cdot \bm_{\rho})\big)$. Thus
\[
\frac{\bg \cdot \bk_{\chi}}{\wtgp(\bg \cdot \bk_{\chi})} 
=\frac{\phi_{k}(\bn \cdot \bm_{\rho})}{\wtgp\big(\phi_{k}(\bn \cdot \bm_{\rho})\big)}
= \frac{\phi_{k}(\bn \cdot \bm_{\rho})}{\phi_{k}\big(\wtga(\bn \cdot \bm_{\rho})\big)}
=\frac{\bn \cdot \bm_{\rho}}{\wtga(\bn \cdot \bm_{\rho})},
\]
and we proved the proposition.
\end{proof}
\begin{Lem}\label{Lem:wt phi2(m mrho) = phi2 (wt (m mrho)}
With the notation as in Proposition~\ref{Prop:From Gamma' to Gamma:S(Gamma)=S(Gamma')}, if $\bm \in \lau$, then
\[
\wtgp\big(
\phi_k(\bm)\big)=\phi_k\big(\wtga(\bm)\big).
\]
\end{Lem}
\begin{proof}
Since $\phi_k(\bm)$ is of the same weight as $\phi_k\big(\wtga(\bm)\big)$ by Lemma~\ref{Lem:phi(m+n)=phi(m)+(n)/phi_k induces a surjective map G dual -> G_k dual}, the assertion follows from the fact that $\phi_k\big(\wtga(\bm)\big)\in\Gamma'$.
\end{proof}
\begin{Lem}
With the notation as in Proposition~\ref{Prop:From Gamma' to Gamma:S(Gamma)=S(Gamma')}, let $\bm$ be the Laurent monomial of weight $j$ in $\Gamma=\left\{\bm \in \lau \st \phi_{k}(\bm) \in \Gamma' \right\}$.
\begin{enumerate}
\item If $k=2$ and $0\leq j<r-a$, then $\phi_{2}(y\cdot \bm) \in \Gamma$.
\item If $k=3$ and $0\leq j<a$, then $\phi_{3}(z\cdot \bm) \in \Gamma$.
\item If $k=1$ and $0\leq j<r-1$, then $\phi_{1}(x\cdot \bm) \in \Gamma$.
\end{enumerate}
\end{Lem}
\begin{proof}
Lemma~\ref{Lem:localizations,lacing} implies the assertion.
\end{proof}

\begin{Prop}\label{Prop:there exists a set of G-bricks}
Let $G$ be the group of type $\frac{1}{r}(1,a,r-a)$ with $r$ coprime to $a$. Let $\mco$ be the set of maximal cones in the fan of the economic resolution $Y$ of $X=\C^3/G$. Then there exists a set $\gr(r,a)$ of $G$-bricks such that there is a bijective map $\mco \rightarrow \gr(r,a)$ sending $\sigma$ to $\Gamma_{\sigma}$ satisfying $S(\Gamma_{\sigma})=\sigma\dual \cap M$. In particular, $U(\Gamma_\sigma)$ is isomorphic to the smooth toric variety $U_{\sigma}=\Spec \C[\sigma\dual \cap M]$ corresponding to $\sigma$.
\end{Prop}
\begin{proof}
From Section~\ref{subsec:G-bricks for 1/r(1,r-1,1)}, the assertion holds when $a=1\text{ or }r-1$. We use induction on $r$ and $a$.

Let $\sigma$ be a 3-dimensional cone in the fan of the economic resolution $Y$ of $X=\C^3/G$. 
For $v=\frac{1}{r}(1,a,r-a)$, we have three cases:
\begin{enumerate}
\item[(1)] the cone $\sigma$ is below the vector $v$.
\item[(2)] the cone $\sigma$ is on the left side of the vector $v$.
\item[(3)] the cone $\sigma$ is on the right side of the vector $v$.
\end{enumerate}

\vskip 2mm
\paragraph*{\bf Case (1) the cone $\sigma$ is below the vector $v$.}
Since there is a unique 3-dimensional cone below $v$, $\sigma=\Cone(v,e_2,e_3)$. 
Consider the central round down function $\phi_{1}$ of the weighted blowup with weight $(1,a,r-a)$. For $\bm=x^{m_1}y^{m_2}z^{m_3} \in \lau$, note that
\[
\phi_{1}(\bm)= \one \quad \text{if and only if} \quad  
m_2=m_3=0 \text{ and } 0 \leq \tfrac{m_1}{r} < 1.
\]
The set $\Gamma:=\phi_{1}\inv(\one)=\{1,x,x^2,\ldots,x^{r-1}\}$ is a $G$-prebrick with the property $S(\Gamma)=\sigma\dual \cap M$. Since the corresponding cone $\sigma(\Gamma)$ is equal to $\sigma$, the $G$-prebrick $\Gamma$ is a $G$-brick.
\vskip 2mm
\paragraph*{\bf Case (2) the cone $\sigma$ is on the left side of $v$.} From the fan of the economic resolution, it follows that $U_2$ is isomorphic to the economic resolution $Y_{2}$ of $\lgr$ with eigencoordinates $\xi,\eta,\zeta$. There exists a unique 3-dimensional cone $\sigma'$ in the toric fan of $Y_{2}$ corresponding to $\sigma$. Let $G_2$ be the group of type $\lgr$. Note that $a$ is strictly less than $r$ so that we can use induction.

Assume that there exists a $G_2$-brick $\Gamma'$ with $S(\Gamma')=(\sigma')\dual \cap M$. By Proposition~\ref{Prop:From Gamma' to Gamma:S(Gamma)=S(Gamma')}, there is a $G$-brick $\Gamma$ with $S(\Gamma)=S(\Gamma')=\sigma\dual \cap M$.
\vskip 2mm
\paragraph{\bf Case (3) the cone $\sigma$ is on the right side of $v$.} The case where the cone $\sigma$ is on the right side of $v$ can be proved similarly.
\end{proof}
\begin{Def}
A $G$-brick $\Gamma$ in $\gr(r,a)$ described above is called a {\em Danilov $G$-brick}. 
\end{Def}
\begin{Prop}\label{Prop:not change deformation space}
With the notation as is in Proposition~\ref{Prop:From Gamma' to Gamma:S(Gamma)=S(Gamma')}, we have $D(\Gamma') \iso D(\Gamma)$. Moreover we have a commutative diagram
\[
\begin{array}{ccc}
U(\Gamma') & \stackrel{\iso}{\longrightarrow} & U(\Gamma)\\[4pt]
\Big\downarrow&&\Big\downarrow\\[4pt]
D(\Gamma') & \stackrel{\iso}{\longrightarrow} & D(\Gamma)
\end{array}
\]
with the vertical morphisms closed embeddings. Therefore for a $G$-brick $\Gamma\in \gr(r,a)$, we have $U(\Gamma)=D(\Gamma)\iso \C^3$.
\end{Prop}
\begin{proof}
Let $\Gamma$ be a $G$-brick and $\Gamma'$ the corresponding $G_k$-brick. Let $\xi,\eta,\zeta$ denote the eigencoordinate for the $G_k$-action. 
From~\eqref{Eqtn:coordinate ring of D(Gamma)}, the coordinate rings of the affine schemes $D(\Gamma)$, $D(\Gamma')$ are
\begin{align*}
\C[D(\Gamma)]&=\C[x_{\rho}, y_{\rho}, z_{\rho} \st \rho \in G\dual]\big{/}I_{\Gamma},\\
\C[D(\Gamma')]&=\C[\xi_{\chi}, \eta_{\chi}, \zeta_{\chi} \st \chi \in G_k\dual]\big{/}I_{\Gamma'}
\end{align*}
where the ideal $I_{\Gamma}$ is $\big\langle\text{the quadrics in \eqref{Eqtn:the conmmutative relations}, $\bbf_{\rho}-1 \st \bbf_{\rho} \in \Lambda(\Gamma)$}\big\rangle$ and the ideal $I_{\Gamma'}$ is $\big\langle \text{the commutative relations, } \bg_{\chi}-1 \st \bg_{\chi} \in \Lambda(\Gamma')\big\rangle$.

By Lemma~\ref{Lem:Connected Gamma}, we have an algebra epimorphism
\[
\mu \colon \C[x_{\rho}, y_{\rho}, z_{\rho} \st \rho \in G\dual] \rightarrow \C[D(\Gamma')] 
\quad \bbf_{\rho} \mapsto\bk_{(\chi)}
\]
defined as follows on the $3r$ generators $\bbf_{\rho}\in\{x_{\rho}, y_{\rho}, z_{\rho}\}$. Let $\bm_{\rho}$ be the unique element of weight $\rho$ in $\Gamma$ and $\chi$ the weight of $\phi_k(\bbf\cdot\bm_{\rho})$. Then $\bk:=\frac{\phi_k(\bbf\cdot\bm_{\rho})}{\phi_k(\bm_{\rho})}$ is a monomial, so $\bk$ induces a linear map $\bk_{(\chi)}$ on the vector space $\C\cdot\phi_k(\bm_{\rho})$.
Then $\mu$ is the morphism sending $\bbf_{\rho}$ to $\bk_{(\chi)}$.
Since the generators of $I_{\Gamma}$ are in $\ker\mu$, $\mu$ induces an epimorphism $\overline{\mu}\colon \C[D(\Gamma)]\rightarrow \C[D(\Gamma')]$.

To construct the inverse of $\overline{\mu}$, first we show that if $\mu(\bbf_{\rho})=\mu(\bbf'_{\rho'})$, then $\bbf_{\rho} \equiv \bbf'_{\rho'} \mod I_{\Gamma}$.
If $\mu(\bbf_{\rho})=\mu(\bbf'_{\rho'})$, then $\phi_k(\bbf\cdot\bm_{\rho})=\phi_k(\bbf'\cdot\bm_{\rho'})$ and 
$\phi_k(\bm_{\rho})=\phi_k(\bm_{\rho'})$. Since both $\bbf_{\rho}$ and $\bbf'_{\rho'}$ are degree of 1, $\bbf=\bbf'$.
By (iv) in Lemma~\ref{Lem:localizations,lacing}, we may assume that $\bm_{\rho'}=\bn\cdot\bm_{\rho}$ for some $\bn\in\mon$.
Since $\phi_k(\bm_{\rho})=\phi_k(\bm_{\rho'})$, $\bn$ induces a linear map equal to 1 on $\bm_{\rho}$, i.e.\ $\bm_{(\rho)}\equiv 1 \mod I_{\Gamma}$ because $\bm_{\rho'}=\bn\cdot\bm_{\rho}$ is a base.
From the following commutative diagram
\[
\begin{array}{ccc}
\C\cdot\bm_{\rho} &\stackrel{\cdot \bn}{\longrightarrow}& \C\cdot\bm_{\rho'}\\
\Big\downarrow {\bbf_{\rho}}&&\Big\downarrow {\bbf_{\rho'}}\\
\C\cdot\wtga(\bbf\cdot\bm_{\rho})&\stackrel{\cdot \bn}{\longrightarrow}& \C\cdot\wtga(\bbf\cdot\bm_{\rho'}),
\end{array}
\]
it suffices to show that $\bn\cdot\wtga(\bbf\cdot\bm_{\rho})$ is a base, which implies that $\bn$ induces a linear map equal to 1 on $\wtga(\bbf\cdot\bm_{\rho})$.
Since
\begin{align*}
\phi_k(\bn \cdot \wtga(\bbf\cdot\bm_{\rho}))
&=\phi_k\Big( \bbf \cdot \bm_{\rho'} \cdot \frac{\wtga(\bbf\cdot\bm_{\rho})}{\bbf\cdot\bm_{\rho}}\Big)\\
&= \phi_k(\bbf\cdot\bm_{\rho'})\cdot \frac{\wtga(\bbf\cdot\bm_{\rho})}{\bbf\cdot\bm_{\rho}}\\
&= \phi_k(\bbf\cdot\bm_{\rho}) \cdot \frac{\wtga(\bbf\cdot\bm_{\rho})}{\bbf\cdot\bm_{\rho}}
=\phi_k\big(\wtga(\bbf\cdot\bm_{\rho})\big),
\end{align*}
the monomial $\bn\cdot\wtga(\bbf\cdot\bm_{\rho})$ is in $\Gamma$ and is a base.

Now define the algebra morphism $\nu\colon \C[\xi_{\chi}, \eta_{\chi}, \zeta_{\chi} \st \chi \in G_k\dual] \rightarrow \C[D(\Gamma)]$ by $\nu(\bg_{\chi})=\bbf_{\rho}$ for $\bg_{\chi}\in\{\xi_{\chi},  \eta_{\chi}, \zeta_{\chi}\}$ so that $\mu(\bbf_{\rho})=\bg_{\chi}$. Since the generators of $I_{\Gamma'}$ are in $\ker\nu$, $\nu$ induces $\overline{\nu}\colon \C[D(\Gamma')]\rightarrow \C[D(\Gamma)]$. 

To show $\overline{\nu}$ is surjective, we prove that generators $\bbf_{\rho}$ are in the image of $\nu$. For $\bbf_{\rho}$ such that $\frac{\phi_k(\bbf\cdot\bm_{\rho})}{\phi_k(\bm_{\rho})}$ is of degree $\leq1$, $\bbf_{\rho}$ is in the image of $\nu$ by definition. By Lemma~\ref{Lem:no change of def spaces-inside} below, it follows that $\overline{\nu}$ is surjective.

Since $\overline{\mu}$ and $\overline{\nu}$ are the inverses of each other, $D(\Gamma)$ is isomorphic to $D(\Gamma')$. Note that $U(\Gamma)=D(\Gamma)\iso \C^3$ for $\Gamma \in \gr(r,1)$ from Section~\ref{subsec:G-bricks for 1/r(1,r-1,1)}. Using induction, we get $D(\Gamma)\iso \C^3$ for $\Gamma \in \gr(r,a)$.
\end{proof}
\begin{Lem}\label{Lem:no change of def spaces-inside}
In the situation as in Proposition~\ref{Prop:not change deformation space}, define
\[
S:=\left\{\bbf_{\rho}\in\{x_{\rho}, y_{\rho}, z_{\rho}\} \st \tfrac{\phi_k(\bbf\cdot\bm_{\rho})}{\phi_k(\bm_{\rho})} \text { is of degree $\leq 1$} \right\}\!.
\]
If $\frac{\phi_k(\bbf\cdot\bm_{\rho})}{\phi_k(\bm_{\rho})}$ is of degree $\geq 2$ for some $\bbf_{\rho}\in\{x_{\rho}, y_{\rho}, z_{\rho}\}$, then $
\bbf_{\rho}$ can be written as a multiple of some elements in $S$ modulo $I_{\Gamma}$.
\end{Lem}
\begin{proof}
We prove this for the left round down function $\phi_2$. Note that  $\frac{\phi_k(y\cdot\bm_{\rho})}{\phi_k(\bm_{\rho})}$ is of degree $\leq 1$ for all ${\rho}\in G\dual$. Thus $y_{\rho}$'s are in $S$.

Suppose that $\frac{\phi_k(\bbf\cdot\bm_{\rho})}{\phi_k(\bm_{\rho})}$ is of degree $\geq 2$ with $\bm_{\rho}=x^{m_1}y^{m_2}z^{m_3}$. Then the monomial $\bbf$ is either $x$ or $z$. In the case where $\bbf=z$, this means that
\[
\frac{1}{r}m_1+\frac{a}{r}m_2+\frac{r-a}{r}m_3-\left\lf\frac{1}{r}m_1+\frac{a}{r}m_2+\frac{r-a}{r}m_3\right\rf \geq \frac{a}{r}.
\]
As in the proof of Lemma~\ref{Lem:Connected Gamma}, there is a positive integer $l$ such that $\phi_{2}(\frac{\bm_{\rho}}{y^{l'}})=\phi_{2}(\bm_{\rho})$ for all $0 \leq l' \leq l$ with $\phi_{2}\big(\frac{\bm_{\rho}}{y^{l+1}} \big)\neq \phi_{2}(\bm_{\rho})$. Note that $\frac{\phi_2(\bbf\cdot\bm_{\rho'})}{\phi_2(\bm_{\rho'})}$ is of degree 1 where $\bm_{\rho'}=\frac{\bm_{\rho}}{y^{l}}$. Thus $\bbf_{\rho'}\in S$. From the commutation relations
\[
\begin{array}{ccc}
\C\cdot\bm_{\rho'} &\stackrel{\cdot y^l}{\longrightarrow}& \C\cdot\bm_{\rho}\\
\Big\downarrow {\bbf_{\rho'}}&&\Big\downarrow {\bbf_{\rho}}\\
\C\cdot\wtga(\bbf\cdot\bm_{\rho'})&\stackrel{\cdot y^l}{\longrightarrow}& \C\cdot\wtga(\bbf\cdot\bm_{\rho}),
\end{array}
\]
since $y^l$ induces a linear map on $\C\cdot\bm_{\rho'}$ set to be 1, we have
\[
\bbf_{\rho} \equiv \bbf_{\rho} \cdot y^l_{(\rho')} \equiv y^l_{(\rho)} \cdot \bbf_{\rho'} \mod I_{\Gamma}.
\]
As all $y_{\rho}$'s are in $S$, the assertion follows.
\end{proof}
\begin{Eg} \label{Eg:Calculating G-graphs in 1/7(1,3,4)} Let $G$ be the group of type $\frac{1}{7}(1,3,4)$ as in Example~\ref{Eg:G graph 1/7(1,3,4) }. The fan of the economic resolution of the quotient variety is shown in Figure~\ref{Fig:Fan of the economic resolution for 1/7(1,3,4)}.

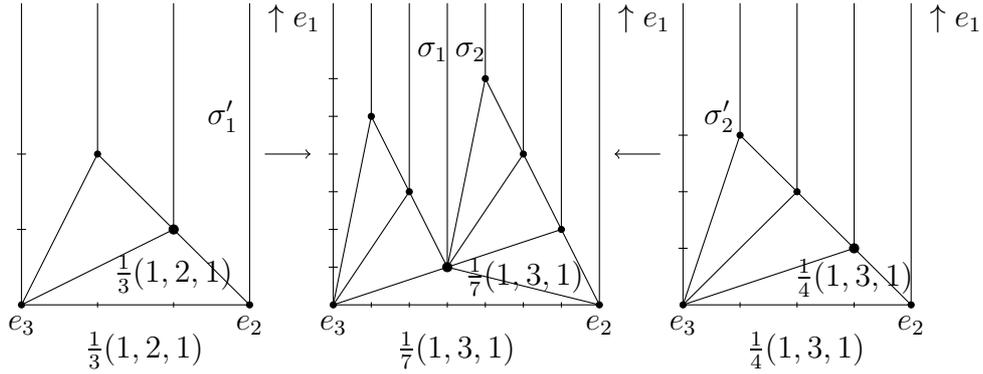
\begin{figure}[h]
\begin{center}
\begin{tikzpicture}
\coordinate [label=below:$e_3$] (e3) at (8.7,0);
\coordinate [label=below:$e_2$] (e2) at (11.7,0);
\foreach \x in {1,2,3}
    \draw (0.75*\x+8.7 ,1pt) -- (0.75*\x+8.7 ,-1pt);
\foreach \y in {1,2,3}
    \draw (8.64,0.75*\y) -- (8.76,0.75*\y);

\foreach \x in {1,2,3} 
\coordinate  (v\x) at (0.75*\x+8.7,3-0.75*\x);
\foreach \x in {1,2,3} 
\draw[fill] (v\x) circle [radius=0.04];

\draw (e3) -- (e2);

\draw (e3) -- (8.7,4);
\draw (e2) -- (11.7,4);

\draw (e2) -- (v1);
\foreach \x in {1,2,3}
\draw (e3) -- (v\x);
\foreach \x in {1,2,3}
\draw (v\x) -- (0.75*\x+8.7,4);

\coordinate [label=below:$e_3$] (e37) at (4.1,0);
\coordinate [label=below:$e_2$] (e27) at (7.6,0);
\foreach \x in {1,2,3,4,5,6}
    \draw (0.5*\x+4.1 ,1pt) -- (0.5*\x+4.1 ,-1pt);
\foreach \y in {1,2,3,4,5,6}
    \draw (4.04,0.5*\y) -- (4.16,0.5*\y);

\foreach \x in {1,2,3} 
\coordinate  (v\x7) at (0.5*\x+4.1,3.5-\x);
\foreach \x in {4,5,6} 
\coordinate  (v\x7) at (0.5*\x+4.1,7-\x);

\foreach \x in {1,2,3,4,5,6} 
\draw[fill] (v\x7) circle [radius=0.04];

\draw (e37) -- (e27);

\draw (e37) -- (4.1,4);
\draw (e27) -- (7.6,4);

\draw (v17) -- (v37);

\foreach \x in {1,2,3}
\draw (e37) -- (v\x7);

\draw (v47) -- (e27);

\foreach \x in {4,5,6}
\draw (v37) -- (v\x7);

\draw (v37) -- (e27);

\foreach \x in {1,2,3,4,5,6}
\draw (v\x7) -- (0.5*\x+4.1,4);

\coordinate [label=below:$e_3$] (e33) at (0,0);
\coordinate [label=below:$e_2$] (e23) at (3,0);
\foreach \x in {1,2}
    \draw (\x ,1pt) -- (\x,-1pt);
\foreach \y in {1,2}
    \draw (-0.06,\y) -- (0.06,\y);

\foreach \x in {1,2} 
\coordinate  (v\x3) at (\x,3-\x);
\foreach \x in {1,2} 
\draw[fill] (v\x3) circle [radius=0.04];

\draw (e33) -- (e23);

\draw (e33) -- (0,4);
\draw (e23) -- (3,4);

\draw (e23) -- (v13);
\foreach \x in {1,2}
\draw (e33) -- (v\x3);
\foreach \x in {1,2}
\draw (v\x3) -- (\x,4);

\draw[fill] (e2) circle [radius=0.04];
\draw[fill] (e3) circle [radius=0.04];
\draw[fill] (e27) circle [radius=0.04];
\draw[fill] (e37) circle [radius=0.04];
\draw[fill] (e23) circle [radius=0.04];
\draw[fill] (e33) circle [radius=0.04];
\draw[->] (3.2,2) -- (3.8,2);
\draw[->] (8.4,2) -- (7.8,2);
\node [right] at (3.1,3.8) {$\uparrow e_1$};    
\node [right] at (7.7,3.8) {$\uparrow e_1$};
\node [right] at (11.8,3.8) {$\uparrow e_1$};

\node [right] at (2.3,2.5) {$\sigma'_{1}$};
\node [right] at (5.06,3.35) {$\sigma_{1}$};
\node [right] at (5.56,3.35) {$\sigma_{2}$};
\node [right] at (8.83,2.5) {$\sigma'_{2}$};
    
\node [right] at (0.7,-0.6) {$\tfrac{1}{3}(1,2,1)$};
\node [right] at (4.8,-0.6) {$\tfrac{1}{7}(1,3,1)$};
\node [right] at (9.4,-0.6) {$\tfrac{1}{4}(1,3,1)$};

\node [below] at (2,0.8) {$\tfrac{1}{3}(1,2,1)$};
\node [right] at (5.7,0.37) {$\tfrac{1}{7}(1,3,1)$};
\node [below] at (v3) {$\tfrac{1}{4}(1,3,1)$};

\draw[fill] (v23) circle [radius=0.06];
\draw[fill] (v37) circle [radius=0.06];
\draw[fill] (v3) circle [radius=0.06];
\end{tikzpicture}
\end{center}
\caption{Recursion process for $\frac{1}{7}(1,3,4)$}\label{Fig:recursion process of 1/7(1,3,4)}
\end{figure}

We now calculate $G$-bricks associated to the following cones:
\begin{align*}
\sigma_1 &:= \Cone\left((1,0,0),\tfrac{1}{7}(1,3,4),\tfrac{1}{7}(3,2,5)\right)\!, \\[3pt]
\sigma_2 &:= \Cone\left((1,0,0),\tfrac{1}{7}(6,4,3),\tfrac{1}{7}(1,3,4)\right)\!.
\end{align*}
Note that the left side of the fan corresponds to the economic resolution for the quotient singularity of type $\frac{1}{3}(1,2,1)$, which is $\GHii{2}$, where $G_2$ is of type $\frac{1}{3}(1,2,1)$. Let $\xi,\eta,\zeta$ denote the eigencoordinates. Let $\sigma'_1$ be the cone in the fan of $\GHii{2}$ which corresponds to $\sigma_1$. Observe that the corresponding $G_2$-brick is
\[
\Gamma'_1 = \big\{ 1, \zeta,\zeta^2 \big\}.
\]
Since the left round down function $\phi_2$ is
\[
\phi_{2}(x^{m_1}y^{m_2}z^{m_3})=\xi^{m_1}\eta^{\lf\frac{1}{7}m_1 + \frac{3}{7} m_2 + \frac{4}{7}m_3\rf}\zeta^{m_3},
\]
the $G$-brick corresponding to $\sigma_1$ is
\begin{align*}
\Gamma_1 &\eqd \big\{ x^{m_1}y^{m_2}z^{m_3} \in \lau \st \phi_{2}(x^{m_1}y^{m_2}z^{m_3}) \in \Gamma'_1 \big\}\\[4pt]
&= \big\{ 1,y,y^2,z,\tfrac{z}{y},\tfrac{z^2}{y},\tfrac{z^2}{y^2} \big\}.
\end{align*}

On the other hand, the right side of the fan corresponds to the economic resolution of the quotient variety $\frac{1}{4}(1,3,1)$ which is $\GHii{3}$, where $G_3$ is of type $\frac{1}{4}(1,3,1)$ with eigencoordinates $\alpha,\beta,\gamma$. Let $\sigma'_2$ be the cone in the fan of $\GHii{2}$ which corresponds to $\sigma_2$. Observe that the corresponding $G_3$-brick is
\[
\Gamma'_2 = \big\{ 1, \beta,\beta^2, \beta^3 \big\}.
\]
Since the right round down function $\phi_3$ is
\[
\phi_{3}(x^{m_1}y^{m_2}z^{m_3})=\alpha^{m_1}\beta^{m_2}\gamma^{\lf\frac{1}{7}m_1 + \frac{3}{7} m_2 + \frac{4}{7}m_3\rf},
\]
the $G$-brick corresponding to $\sigma_2$ is
\begin{align*}
\Gamma_2 &\eqd \big\{ x^{m_1}y^{m_2}z^{m_3} \in \lau \st \phi_{2}(x^{m_1}y^{m_2}z^{m_3}) \in \Gamma'_2 \big\}\\[4pt]
&= \big\{ 1,z,y,y^2,\tfrac{y^2}{z},\tfrac{y^3}{z},\tfrac{y^3}{z^2} \big\}.
\end{align*}
From Example~\ref{Eg:cone assoc to G graph 1/7(1,3,4) },
$\sigma(\Gamma_1)=\sigma_1$ and $\sigma(\Gamma_2)=\sigma_2$.\eeg
\subsection{Stability parameters for $\gr(r,a)$}\label{subsec:Our stability parameters}
Let $G \subset \GL_3(\C)$ be the finite subgroup of type $\frac{1}{r}(1,a,r-a)$ with $r$ coprime to $a$. We may assume $2 a<r$. Let $G_2$ and $G_3$ be the groups of type $\lgr$ and of type $\rgr$, respectively. 

Given stability conditions $\theta^{(2)}$ for Danilov $G_2$-bricks and $\theta^{(3)}$ for Danilov $G_3$-bricks, 
take a GIT parameter $\theta_P\in\Theta$ satisfying the following system of linear equations:
\begin{equation}\label{Eqtn:linear equations}
\begin{cases}
\theta^{(2)}(\chi) = \sum\limits_{\phi_2(\rho)=\chi} \theta_P(\rho) &\text{ for all $\chi \in G_2\dual$,}\\[3pt]
\theta^{(3)}(\chi') = \sum\limits_{\phi_3(\rho)=\chi'} \theta_P(\rho) &\text{ for all $\chi' \in G_3\dual$.}
\end{cases}
\end{equation}
Define the GIT parameter $\vtheta \in \Theta$ by
\begin{equation}\label{Eqtn:def. of vartheta}
\vtheta(\rho) =
\begin{cases}
-1 &\text{if $0 \leq \wt(\rho)<a$ ,}\\
1 &\text{if $r-a \leq \wt(\rho)<r$,}\\
0 &\text{otherwise.}\\
\end{cases}
\end{equation}
Observe that $\sum\limits_{\phi_k(\rho)=\chi} \vtheta(\rho)=0$ for all $\chi \in G_k\dual$
\footnote{In addition, if any $\theta \in \Theta$ satisfies that $\sum_{\phi_k(\rho)=\chi} \theta(\rho)=0$ for all $\chi \in G_k\dual$ and $k=2,3$, then $\theta$ must be a constant multiple of $\vtheta$. This also explains the existence of a solution $\theta_P $ for \eqref{Eqtn:linear equations}.}.
For a sufficiently large natural number $m$, set
\begin{equation}\label{Eqtn:stability parameter m large}
\theta := \theta_P + m \vtheta.
\end{equation}

We claim that every $\Gamma\in\gr(r,a)$ is $\theta$-stable.
\begin{Eg}
As in Example~\ref{Eg:Calculating G-graphs in 1/7(1,3,4)}, let $G$ be the group of type $\frac{1}{7}(1,3,4)$. For each $0\leq i \leq 6$, let $\rho_i$ denote the irreducible representation of $G$ whose weight is $i$. We saw that the left side of the fan is $\GHii{2}$, where $G_2$ is of type $\frac{1}{3}(1,2,1)$ and that the right side of the fan is $\GHii{3}$, where $G_3$ is of type $\frac{1}{4}(1,3,1)$. Let $\{\chi_0,\chi_1,\chi_2\}$ and $\{\chi'_0,\chi'_1,\chi'_2,\chi'_3\}$ be the characters of $G_2$ and $G_3$, respectively. Take GIT parameters $\theta^{(2)}$, $\theta^{(3)}$ corresponding to $\GHil$ such as (see \eqref{Eqtn:Stab for G-Hilb}):
\[
\theta^{(2)}=(-2,1,1), \quad \theta^{(3)}=(-3,1,1,1).
\]
We have the following system of linear equations:
\[
\left\{
\begin{array}{ccl}
-2 & = & \theta_P(\rho_0)+\theta_P(\rho_3)+\theta_P(\rho_6), \\
1 &=& \theta_P(\rho_1)+\theta_P(\rho_4),\\
1 &=&\theta_P(\rho_2)+\theta_P(\rho_5),\\
-3 &=&\theta_P(\rho_0)+\theta_P(\rho_4),\\
1 &=&\theta_P(\rho_1)+\theta_P(\rho_5),\\
1 &=&\theta_P(\rho_2)+\theta_P(\rho_6),\\
1 &=&\theta_P(\rho_3).
\end{array}
\right.
\]
Take $\theta_P=(-1,3,3,1,-2,-2,-2)$ as a partial solution. For the parameter $\vtheta=(-1,-1,-1,0,1,1,1)$, define $\theta=\theta_P+m\vtheta$ for large $m$. 

Consider the following $G$-brick
\[
\Gamma=\Big\{ 1,y,y^2,z,\tfrac{z}{y},\tfrac{z^2}{y},\tfrac{z^2}{y^2} \Big\}.
\]

Let $\sF$ be the submodule of $C(\Gamma)$ with basis $A=\big\{z,\tfrac{z}{y},\tfrac{z^2}{y}\big\}$. Note that $\vtheta(\sF)>0$ and 
\[
\phi_2\inv\big(\phi_2(A)\big)=\Big\{z,\tfrac{z}{y},\tfrac{z^2}{y},\tfrac{z^2}{y^2} \Big\} \supsetneq A.
\]
Thus $\theta(\sF)$ is positive for large enough $m$. More precisely,
\[
\theta(\sF)=3-m+(-2+m)+(-2+m)=m-1
\]
is positive if $m>1$.

On the other hand, consider the submodule $\sG$ of $C(\Gamma)$ with basis $B=\big\{\tfrac{z}{y},\tfrac{z^2}{y}\big\}$. Note that $\vtheta(\sG)=0$ and $\phi_2\inv\big(\phi_2(B)\big)=B$. In this case, the set $\phi_2(B)$ gives a submodule $\sG'$ of $C(\Gamma')$ with 
\[
\theta^{(2)}(\sG')=\theta(\sG).
\]
Since $C(\Gamma')$ is $\theta^{(2)}$-stable, $\theta^{(2)}(\sG')$ is positive. Hence $\theta(\sG)$ is positive.\eeg
\begin{Lem} \label{Lem:our stability} Let $\theta$ be the parameter in \eqref{Eqtn:stability parameter m large}.
For the set $\gr(r,a)$ in Proposition~\ref{Prop:there exists a set of G-bricks}, if $\Gamma$ is in $\gr(r,a)$, then $\Gamma$ is $\theta$-stable.
\end{Lem}
\begin{proof} Let $\Gamma$ be a $G$-brick in $\gr$ and $\sigma$ the cone corresponding to $\Gamma$. We have the following three cases as in Section~\ref{subsec:G-bricks for 1/r(1,a,r-a)}:
\begin{enumerate}
\item[(1)] the cone $\sigma$ is below the vector $v$.
\item[(2)] the cone $\sigma$ is on the left side of the vector $v$.
\item[(3)] the cone $\sigma$ is on the right side of the vector $v$.
\end{enumerate}

In Case (1), $\Gamma=\{1,x,x^2,\ldots,x^{r-2},x^{r-1}\}$.
By Lemma~\ref{Lem:combinatorial description of submodule}, any nonzero proper submodule $\sG$ of $C(\Gamma)$ is given by
\[
A=\{x^{j},x^{j+1}, \ldots,x^{r-2},x^{r-1}\}
\]
for some $1\leq j \leq r-1$. Since $\vtheta(\sG)>0$, $\Gamma$ is $\theta$-stable for sufficiently large $m$.

We now consider Case (2). 
Let $\Gamma'$ be the $G_2$-brick corresponding to $\Gamma$.
Let $\sG$ be a submodule of $C(\Gamma)$ with $\C$-basis $A \subset \Gamma$. Lemma~\ref{Lem:localizations,lacing} and Lemma~\ref{Lem:combinatorial description of submodule} imply that if $\bm_{\rho} \in A$ for $0 \leq \wt(\bm_{\rho}) <a$, then $\phi_2\inv\big(\phi_2(\bm_{\rho})\big) \subset A$. Thus $\vtheta(\sG) \geq 0$ from the definition of $\vtheta$.

If $\vtheta(\sG) > 0$, then it follows that $\theta(\sG)>0$ for sufficiently large $m$. 

Let us assume that $\vtheta(\sG) = 0$. Note that $A=\phi_2\inv\big(\phi_2(A)\big)$; otherwise there exists $\bm_{\rho}$ in $\phi_2\inv\big(\phi_2(A)\big) \setminus A$ with $0 \leq \wt(\bm_{\rho}) <a$. 
To show that $\theta(\sG)$ is positive, we prove that $\phi_2(A)$ gives a submodule $\sG'$ of $C(\Gamma')$ and that $\theta(\sG)=\theta^{(2)}(\sG')$. Since $\theta$ satisfies the equations~\eqref{Eqtn:linear equations}, it suffices to show that $\phi_2(A)$ gives a submodule of $C(\Gamma')$. Let $\xi, \eta, \zeta$ be the coordinates of $\C^3$ with respect to the action of $G_2$. By Lemma~\ref{Lem:combinatorial description of submodule}, it is enough to show 
that if $\bg \cdot \phi_2(\bm_{\rho}) \in \Gamma'$ for some $\bg \in \{\xi, \eta, \zeta\}$ and $\bm_{\rho} \in A$, then $\bg \cdot \phi_2(\bm_{\rho})\in\phi_2(A)$. Suppose that $\bg \cdot \phi_2(\bm_{\rho}) \in \Gamma'$ for some $\bm_{\rho} \in A$. By Lemma~\ref{Lem:Connected Gamma}, there exists $\bm_{\rho'}$ such that 
\[
\phi_2(\bbf \cdot \bm_{\rho'})=\bg\cdot\phi_2(\bm_{\rho})
\]
with $\phi_2(\bm_{\rho'})=\phi_2(\bm_{\rho})$ for some $\bbf\in\{x,y,z\}$. In particular, $\bbf \cdot \bm_{\rho'}$ is in $ \Gamma$. Since $A=\phi_2\inv\big(\phi_2(A)\big)$, we have $\bm_{\rho'} \in A$, which implies $\bbf \cdot \bm_{\rho'}\in A$ as $A$ is a $\C$-basis of $\sG$. Thus $\bg\cdot\phi_2(\bm_{\rho})$ is in $\phi_2(A)$.
\end{proof}
\begin{Rem}
At this moment, our stability parameter $\theta$ in \eqref{Eqtn:stability parameter m large} has nothing to do with {\em K\k{e}dzierski's GIT chamber} $\wc(r,a)$ described in~\cite{Ked14}. In Section~\ref{Sec:Kedzierski's GIT chamber}, it is shown that the parameter $\theta$ is in $\wc(r,a)$. \erem
\subsection{Main Theorem}\label{subsec:Main Theorem}
\begin{Thm}\label{Thm:Main Theorem}
The economic resolution $Y$ of a 3-fold terminal quotient singularity $X=\C^3/G$ is isomorphic to the birational component $\yth$ of the moduli space $\mth$ of $\theta$-stable $G$-constellations for a suitable parameter~$\theta$.
\end{Thm}
\begin{proof}
From Proposition~\ref{Prop:there exists a set of G-bricks} and Lemma~\ref{Lem:our stability}, Proposition~\ref{Prop:open immersion} implies that there exists an open immersion from $Y$ to $\yth$ fitting in the following commutative diagram:
\[
\begin{array}{rcc}
Y& \rightarrow & \yth \\[4pt]
&\searrow & \downarrow \\[4pt]
&&X.
\end{array}
\]
Since both $Y$ and $\yth$ are projective over $X$, the open immersion $Y \rightarrow  \yth$ is a closed embedding. As both $Y$ and $\yth$ are 3-dimensional and irreducible, this embedding is an isomorphism.
\end{proof}
\begin{Con}\label{Con:M_theta irreducible}
The moduli space $\mth$ is irreducible.
\end{Con}
Proposition~\ref{Prop:not change deformation space} implies that the irreducible component $\yth$ is actually a connected component. In addition, if every torus invariant $\theta$-stable $G$-constellation lies over the birational component $\yth$, then $\mth$ is irreducible. 
For $a=2$, we can prove Conjecture~\ref{Con:M_theta irreducible} so the economic resolution is isomorphic to $\mth$ for $\theta\in\gr(r,a)$ (See \cite{Thesis}). We hope to establish this more generally in future work.
\begin{Rem}
By construction, 
$\mzero=\Spec \C[\rg]^{\gld}$ is the moduli space of ${0}$-semistable $G$-constellations up to $S$-equivalence. Since there exists an algebra isomorphism
$
\C[\rg]^{\gld} \rightarrow \C[x,y,z]^G$, $\mzero$ is isomorphic to $\C^3/G$. In particular, $\mzero$ is irreducible.\erem
\section{K\k{e}dzierski's GIT chamber}\label{Sec:Kedzierski's GIT chamber}
K\k{e}dzierski\cite{Ked14} described his GIT cone in $\Theta$ using a set of inequalities. Using his lemma, we can prove further that the cone is actually a GIT chamber $\wc$. 
In this section, we provide a description of $\wc$ using the $A_{r-1}$ root system. 
Define
\[
\grx=\{\Gamma\in \gr(r,a) \st x\not\in\Gamma \}.
\]
\subsubsection*{\bf K\k{e}dzierski's lemma}
By the same argument as in Lemma 6.7 of \cite{Ked14}, we can prove that it suffices to check the $\theta$-stability for $G$-bricks $\Gamma$ not containing $x$.
\begin{Lem}[K\k{e}dzierski's lemma\cite{Ked14}]\label{Lem:it suffices to check for Gamma without x}
For a parameter $\theta\in \Theta$, the following are equivalent.
\begin{enumerate}
\item Every $\Gamma \in \gra$ is $\theta$-stable. 
\item Every $\Gamma \in \grx$ is $\theta$-stable.
\end{enumerate}
\end{Lem}
Let $A$ be the finite group of type $\frac{1}{r}(a,r-a)$. Since $A\iso G$ as groups, the GIT parameter space $\Theta$ of $G$-constellations can be canonically identified with that of $A$-constellations.

Since $G$-constellations which $x$ acts trivially on are supported on the hyperplane $(x=0)\subset \C^3$, they 
can be considered as $A$-constellations. As $\Gamma \in \grx$ is the set of $G$-bricks corresponding to $G$-constellations supported on $(x=0)\subset \C^3$, Lemma~\ref{Lem:it suffices to check for Gamma without x} implies that the GIT chamber for $\gr(r,a)$ is equal to a GIT chamber of $A$-constellations.
\subsubsection*{\bf K\k{e}dzierski's GIT chamber}
We describe a set of simple roots $\Delta$ so that $\yth$ is isomorphic to the economic resolution for $\theta \in \wc(\Delta)$. After considering the case of $a=1$, we describe simple roots for the case of $\frac{1}{r}(1,a,r-a)$ using a recursion process.
\subsubsection*{Root system $A_{r-1}$}\label{subsec:root system}
Identify $I:=\Irr(G)$ with $\Z/r\Z$.
Let $\left\{\varepsilon_{i}\st i \in I \right\}$ be an orthonormal basis of $\Q^r$, i.e.\ $\langle\varepsilon_i,\varepsilon_j\rangle=\delta_{ij}$. 
Define 
\[
\Phi:=\{\varepsilon_i-\varepsilon_j \st i,j \in I, i\neq j \}.
\] 
Let $\hd$ be the subspace of $\Q^r$ generated by $\Phi$.
Elements in $\Phi$ are called {\em roots}.

For each nonzero $i \in I$,
set $\alpha_i=\varepsilon_i-\varepsilon_{i-a}$. 
Let $\rho_i$ denote the irreducible representation of $G$ of weight $i$. Note that each root $\alpha$ can be considered as the support of a submodule of a $G$-constellation. In other words, $\alpha_i$ corresponds to the dimension vector of $\rho_i$. In general we consider a root $\alpha=\sum_{i} n_i \alpha_i$ as the dimension vector of the representation $\oplus n_i \rho_{i}$. Abusing notation, let $\alpha=\sum_{i} n_i \alpha_i$ denote the corresponding representation $\oplus n_i \rho_{i}$.

Let $\Delta$ be a set of simple roots. Define $\wc(\Delta)\subset\Theta$ associated to $\Delta$ as
\[
\wc(\Delta):=\{\theta \in \Theta \st \theta(\alpha) >0 \quad \forall \alpha \in \Delta\}.
\]
Note that for the cone $\Theta_{+}$ for $\GHil$ in \eqref{Eqtn:Stab for G-Hilb}, the corresponding set of simple roots is
\[
\Delta_+=\{\varepsilon_{i}-\varepsilon_{i-a} \in \Phi \st i \in I, i \neq 0 \}=\{\alpha_i \st i \in I, i \neq 0 \}.
\]
\subsubsection*{The case of $\frac{1}{r}(1,r-1,1)$.}
From Theorem~\ref{Thm:Ked. a=1 smooth}, we know that the economic resolution of $X=\C^3/G$ is isomorphic to $\GHilb{3}$ if $G$ is of type $\frac{1}{r}(1,r-1,1)$. Thus in this case, the $G$-bricks are just Nakamura's $G$-graphs, which are $\theta$-stable for $\theta \in \Theta_+$,
where
\[
\Theta_{+} := \left\{\theta \in \Theta \st \theta\left(\rho\right) >0 \text{ for } \rho \neq \rho_0 \right\}.
\]
The corresponding set of simple roots is
\begin{align*}
\Delta&=\left\{\varepsilon_{i}-\varepsilon_{i+1} \in \Phi \st i \in I, i \neq 0 \right\}\\
&=\left\{\varepsilon_{1}-\varepsilon_{2},\varepsilon_{2}-\varepsilon_{3},\ldots, \varepsilon_{r-1}-\varepsilon_{0}\right\}.
\end{align*}

\begin{Eg}\label{Eg:Simple roots 1/3(1,2,1) and 1/4(1,3,1)}
For the group of type $\frac{1}{3}(1,2,1)$, let $\left\{\varepsilon^L_j\st j=0,1,2\right\}$ be the standard basis of $\Q^3$. The corresponding set of simple roots is
\[
\Delta^L=\left\{\varepsilon^L_1-\varepsilon^L_2,\varepsilon^L_2-\varepsilon^L_0\right\}.
\]

Similarly, for the group of type $\frac{1}{4}(1,3,1)$ with $\left\{\varepsilon^R_k\st k=0,1,2,3\right\}$ the standard basis of $\Q^4$,
\[
\Delta^R=\{\varepsilon^R_1-\varepsilon^R_2,\varepsilon^R_2-\varepsilon^R_3,
\varepsilon^R_3-\varepsilon^R_0\}
\]
is the corresponding set of simple roots for type $\frac{1}{4}(1,3,1)$.
\eeg
\vskip 2mm
\subsubsection*{The case of $\frac{1}{r}(1,a,r-a)$.}
Let $G$ be the group of type $\frac{1}{r}(1,a,r-a)$. Let $\Delta^L$ and $\Delta^R$ denote the sets of simple roots for the types of $\lgr$ and of $\rgr$, respectively. 
As in Section~\ref{subsec:root system}, let 
\begin{align*}
\left\{\varepsilon^L_l \st l=0,1,\ldots,a-1\}, \quad \{\varepsilon^R_k \st k=0,1,\ldots,r-a-1\right\}
\end{align*}
be the standard basis of $\Q^a$ and $\Q^{r-a}$, respectively.

From the two sets $\Delta^L$ and $\Delta^R$, we construct a set $\Delta$ of simple roots in $A_{r-1}$ as follows.
First, as in Section~\ref{subsec:root system}, let the standard basis $\left\{\varepsilon_i\st i \in I \right\}$ of $\Q^r$ be identified with the union of the two sets 
\begin{align*}
\left\{\varepsilon^L_l \st l=0,1,\ldots,a-1\} \text{ and } \{\varepsilon^R_k \st k=0,1,\ldots,r-a-1\right\}
\end{align*}
using the following identification:
\begin{equation}\label{Eqtn:identification of basis}
\left\{
\begin{array}{llll}
\varepsilon^L_l &= \varepsilon_i &\text{ with } i \equiv l \!\!\! \mod a &\text{if } r-a \leq i <r,\\
\varepsilon^R_k &= \varepsilon_i &\text{ with } i \equiv k \!\!\!\! \mod (r\!-\!a)  &\text{if } 0 \leq i <r-a.
\end{array}
\right.
\end{equation}
With the identification above, define $\Delta$ to be
\begin{equation}\label{Eqtn:Def.Admissible roots}
\Delta = \Delta^L \cup \{\varepsilon_{\lf \frac{r-1}{a}\rf a}-\varepsilon_{\lf\frac{r-1}{r-a}\rf(r-a)-a}\}\cup \Delta^R .
\end{equation}
The root $\varepsilon_{\lf \frac{r-1}{a}\rf a}-\varepsilon_{\lf\frac{r-1}{r-a}\rf(r-a)-a}$ is called the {\em added root} in $\Delta$.
Note that $\Delta$ is actually a set of simple roots in $A_{r-1}$.
\begin{Def}
With $\Delta$ as above, the corresponding Weyl chamber 
\[
\wc(r,a):=\wc(\Delta)=\{\theta \in \Theta \st \theta(\alpha) >0 \quad \forall \alpha \in \Delta \}
\]
is called {\em K\k{e}dzierski's GIT chamber} for $G=\frac{1}{r}(1,a,r-a)$.
\end{Def}
\begin{Prop} Let $\wca$ be K\k{e}dzierski's GIT chamber.\begin{enumerate}
\item The parameter $\vtheta$ in \eqref{Eqtn:def. of vartheta} is a ray of $\wc(r,a)$.
\item Any $G$-brick in $\gr(r,a)$ is $\theta$-stable for $\theta\in\wca$.
\item The cone $\wc(r,a)$ is a full GIT chamber.
\end{enumerate}
\end{Prop}
\begin{proof} 
We may assume $a<r-a$. First, by construction, $\vtheta$ is zero on the sets $\Delta^L$ and $\Delta^R$ with the identification \eqref{Eqtn:Def.Admissible roots}. To prove (i), it remains to show that $\vtheta(\alpha)$ is positive where $\alpha$ is the added root in $\Delta$. Since
\[
\alpha=\varepsilon_{\lf \frac{r-1}{a}\rf a}-\varepsilon_{\lf\frac{r-1}{r-a}\rf(r-a)-a}=\sum_{\phi_2(\rho_i) =\chi_0} \alpha_i +\alpha_{r-a},
\]
where $\chi_0$ is the trivial representation of $G_2$, (i) follows.

For $\theta$ defined by~\eqref{Eqtn:stability parameter m large}, every $\Gamma \in \grx$ is $\theta$-stable. For the group $A$ of type $\frac{1}{r}(a,-a)$, Kronheimer\cite{K89} showed that the chamber structure of the GIT parameter space of $A$-constellations is the same as the Weyl chamber structure of $A_{r-1}$\footnote{For an explicit description, see Section 5.1 in \cite{Thesis}}. Thus for $\grx$ considered as $A$-constellations, we have a Weyl chamber of the $A_{r-1}$ root system containing the parameter $\theta$. 

By K\k{e}dzierski's lemma, to prove (ii), it suffices to show that $\wc(r,a)$ contains the parameter $\theta$.
Observe that every parameter in $\wca$ satisfies the system of equations \eqref{Eqtn:linear equations} for some $\theta^{(2)}\in \wc(a,-r)$ and $\theta^{(3)}\in\wc(r-a,r)$ by construction. Since $\vtheta$ in \eqref{Eqtn:def. of vartheta} is a ray of the chamber $\wca$, it follows that $\theta \in \wca$.

It remains to prove (iii). By considering $G$-constellations supported on the hyperplane $(x=0)\subset \C^3$, it follows that any facet of $\wca$ is an actual GIT wall in $\Theta$. Therefore K\k{e}dzierski's GIT chamber $\wca$ is a full GIT chamber in the stability parameter space $\Theta$ (see \cite{Thesis, Chambers}). 
\end{proof}
\begin{Prop}\label{Prop:generated by a elements} 
Assume that $a < r-a$. Let $\theta$ be an element in $\wca$. Then $\theta(\alpha_i)$ is negative if and only if $0 \leq i <a$. Thus any $\theta$-stable $G$-constellation is generated by $\rho_{0}, \rho_{1}, \ldots, \rho_{a-1}$.
\end{Prop}
\begin{proof}
Let $\Delta$ be the set of simple roots corresponding to $\wca$. Recall that any positive sum of simple roots is positive on $\theta$.

Suppose that $0 \leq i <a$. From the identification~\eqref{Eqtn:identification of basis}, note that $\varepsilon_{i}$ is identified with $\varepsilon^R_k$ for some $k$ and that $\varepsilon_{i-a}=\varepsilon_{i+(r-a)}$ is identified with $\varepsilon^L_l$ for some $l$.
Note that $\varepsilon_{\lf \frac{r-1}{a}\rf a}$ is identified with a vector $\varepsilon^L$ and that $\varepsilon_{\lf\frac{r-1}{r-a}\rf(r-a)-a}$ is identified with a vector $\varepsilon^R$. Since we added the root $\varepsilon_{\lf \frac{r-1}{a}\rf a}-\varepsilon_{\lf\frac{r-1}{r-a}\rf(r-a)-a}$ to $\Delta$, the root $\alpha_i =\varepsilon_{i}-\varepsilon_{i-a}=\varepsilon^R_k-\varepsilon^L_l$ is a negative sum of simple roots in $\Delta$.

Suppose that $a \leq i <r-a$. The root $\alpha_i=\varepsilon_{i}-\varepsilon_{i-a}$ is a sum of simple roots in $\Delta^R$. A recursive argument yields that $\alpha_i$ is a positive sum of simple roots in $\Delta^R$. Thus $\alpha_i$ is a positive sum of simple roots in $\Delta$.

Consider the case where $r-a \leq i <r$ and the root $\alpha_i=\varepsilon_{i}-\varepsilon_{i-a}$.
From the identification~\eqref{Eqtn:identification of basis}, $\varepsilon_{i}$ is identified with $\varepsilon^L_k$ for some $k$ and $\varepsilon_{i-a}$ is identified with $\varepsilon^R_l$ for some $l$.
Thus $\alpha_i=\varepsilon^L_k-\varepsilon^R_l$ is a positive sum of simple roots in $\Delta$ with the same reason as the case where $0 \leq i <a$.
\end{proof}
\begin{Eg}\label{Eg:Simple roots 1/7(1,3,4)}
Let $G$ be the group of type $\frac{1}{7}(1,3,4)$. From the fan of the economic resolution of this case (see Example~\ref{Eg:econ. resolns of 1/7(1,3,4)}), the left and right sides are the economic resolutions of 
singularities of $\frac{1}{3}(1,2,1)$ and $\frac{1}{4}(1,3,1)$, respectively. By Example~\ref{Eg:Simple roots 1/3(1,2,1) and 1/4(1,3,1)}, we have two sets
\[
\Delta^L=\{\varepsilon^L_1-\varepsilon^L_2,\varepsilon^L_2-\varepsilon^L_0\} \text{ and }
\Delta^R=\{\varepsilon^R_1-\varepsilon^R_2,\varepsilon^R_2-\varepsilon^R_3,
\varepsilon^R_3-\varepsilon^R_0\}.
\]
As in the construction~\eqref{Eqtn:Def.Admissible roots}, the corresponding set of simple roots is
\begin{align*}
\Delta&=\{\varepsilon_4-\varepsilon_5,\varepsilon_5-\varepsilon_6,\underline{\varepsilon_6-\varepsilon_1},\varepsilon_1-\varepsilon_2,\varepsilon_2-\varepsilon_3,\varepsilon_3-\varepsilon_0\}\\
&=\{\alpha_4+\alpha_1, \alpha_5+\alpha_2, \underline{-\alpha_1-\alpha_5-\alpha_2}, \alpha_1+\alpha_5,\alpha_2+\alpha_6, \alpha_3 \},
\end{align*}
where the underlined root is the added root as in~\eqref{Eqtn:Def.Admissible roots}.
Thus the set of parameters $\theta\in\Theta$ satisfying 
\[
\begin{array}{ccc}
\theta(\rho_4\oplus\rho_1)>0, &\theta(\rho_5\oplus\rho_2)>0, &\theta(\rho_1\oplus\rho_5\oplus\rho_2)<0, \\
\theta(\rho_1\oplus\rho_5)>0, &\theta(\rho_2\oplus\rho_6)>0, &\theta(\rho_3)>0 
\end{array}
\]
is K\k{e}dzierski's GIT chamber $\wca$ where $\rho_i$ is the irreducible representation of $G$ of weight $i$.

The rays of the chamber $\wca$ are the row vectors of the matrix
\[
\left(
\begin{array}{rrrrrrr}
-1&0&0&1&0&0&0\\
-1&0&0&0&0&0&1\\
-1&0&-1&0&0&1&1\\
-1&-1&-1&0&1&1&1\\
-1&-1&0&0&1&1&0\\
-1&0&0&0&1&0&0
\end{array}
\right)
\]
with the dual basis $\{\theta_i\}$ with respect to $\{\rho_i\}$.
Observe that for any $\theta \in \wca$, $\theta(\rho_i)$ is negative if and only if $0 \leq i <3$.\eeg

\section{Example: type $\frac{1}{12}(1,7,5)$}\label{Sec:example for 1/12(1,7,5)}
\begin{figure}
\begin{center}
\begin{tikzpicture}
\coordinate [label=left:$e_3$] (e3) at (0,0);
\coordinate [label=right:$e_2$] (e2) at (12,0);
\foreach \x in {1,2,3,4,5,6,7,8,9,10,11}
    \draw (\x ,2pt) -- (\x ,-2pt);
\foreach \y in {1,2,3,4,5,6,7,8,9,10,11}
    \draw (2pt,\y) -- (-2pt,\y);    

\coordinate [label=left:$v_1$] (v1) at (1,7);
\draw[fill] (v1) circle [radius=0.05];
\coordinate [label=below:$v_2$] (v2) at (2,2);
\draw[fill] (v2) circle [radius=0.05];
\coordinate [label=left:$v_3$] (v3) at (3,9);
\draw[fill] (v3) circle [radius=0.05];
\coordinate [label=below:$v_4$] (v4) at (4,4);
\draw[fill] (v4) circle [radius=0.05];
\coordinate [label=left:$v_5$] (v5) at (5,11);
\draw[fill] (v5) circle [radius=0.05];
\coordinate [label=below right:$v_6$] (v6) at (6,6);
\draw[fill] (v6) circle [radius=0.05];
\coordinate [label=below:$v_7$] (v7) at (7,1);
\draw[fill] (v7) circle [radius=0.05];
\coordinate [label=below left:$v_8$] (v8) at (8,8);
\draw[fill] (v8) circle [radius=0.05];
\coordinate [label=below:$v_9$] (v9) at (9,3);
\draw[fill] (v9) circle [radius=0.05];
\coordinate [label=above left:$v_{10}$] (v10) at (10,10);
\draw[fill] (v10) circle [radius=0.05];
\coordinate [label=below right:$v_{11}$] (v11) at (11,5);
\draw[fill] (v11) circle [radius=0.05];
\foreach \t in {2,3}
\draw[fill] (e\t) circle [radius=0.05];
\draw (e3) -- (e2);
\draw (e3) -- (0,12);
\draw (e2) -- (12,12);
\draw (e3) -- (v1);
\draw (e3) -- (v6);
\draw (v6) -- (v7);
\draw (e3) -- (v7);
\draw (v7) -- (e2);
\draw (v1) -- (v2);
\draw (v2) -- (v3);
\draw (v3) -- (v4);
\draw (v4) -- (v5);
\draw (v5) -- (v6);
\draw (v6) -- (v7);
\draw (v7) -- (v8);
\draw (v7) -- (v11);
\draw (v9) -- (e2);
\draw (v8) -- (v9);
\draw (v9) -- (v10);
\draw (v10) -- (v11);
\draw (v11) -- (e2);
\draw (v4) -- (v7);
\draw (v2) -- (v7);
\foreach \z in {1,2,3,4,5,6,7,8,9,10,11}
\draw (v\z) -- (\z,12);
\foreach \t in {1,2,3,4,5,6,7,8,9,10,11,12}
\node [right] at (12.2-\t,9.8) {$\sigma_{\t}$};
\node [right] at (10.5,3) {$\tau_{1}$};
\node [right] at (9.6,6) {$\tau_{2}$};
\node [right] at (7.8,4.5) {$\tau_{3}$};
\node [right] at (9,1.5) {$\tau_{4}$};
\node [right] at (5.5,3.5) {$\tau_{5}$};
\node [right] at (4.7,6.5) {$\tau_{6}$};
\node [right] at (2.7,4.3) {$\tau_{7}$};
\node [right] at (3.9,2.3) {$\tau_{8}$};
\node [right] at (0.7,3) {$\tau_{9}$};
\node [right] at (3,1) {$\tau_{10}$};
\node [right] at (6,0.5) {$\tau_{0}$};
\node [right] at (12.1,11.8) {$\uparrow e_1$};
\end{tikzpicture}
\end{center}
\caption{Toric fan of the economic resolution for $\frac{1}{12}(1,7,5)$}\label{Fig:Toric Fan of 1/12(1,7,5)}
\end{figure}
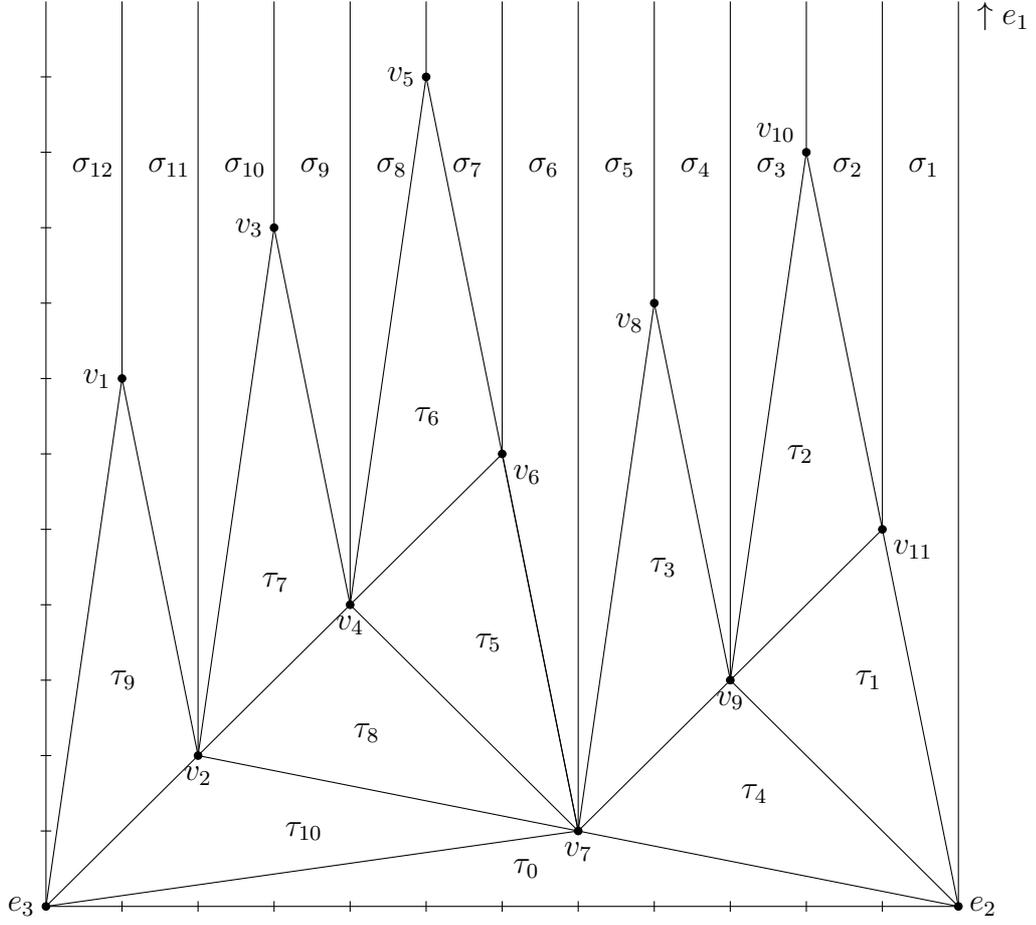
In this section, as a concrete example, we calculate Danilov $G$-bricks and the corresponding set of simple roots $\sra$ for the group $G$ of type $\frac{1}{12}(1,7,5)$.

Let $G$ be the finite group of type $\frac{1}{12}(1,7,5)$ with eigencoordinates $x,y,z$ and $L$ the lattice $L=\Z^3+\Z\cdot \frac{1}{12}(1,7,5)$. Let $X$ denote the quotient variety $\C^3/G$ and $Y$ the economic resolution of $X$. The toric fan $\Sigma$ of $Y$ is shown in Figure~\ref{Fig:Toric Fan of 1/12(1,7,5)}.

To use the recursion process in Section~\ref{Sec:Main Theorem}, first we need to investigate the cases of type $\frac{1}{7}(1,2,5)$ and of type $\frac{1}{5}(1,2,3)$. Let $G_2$ be the group of type $\frac{1}{7}(1,2,5)$ with eigencoordinates $\xi_2,\eta_2,\zeta_2$ and $G_3$ be the group of type $\frac{1}{5}(1,2,3)$ with eigencoordinates $\xi_3,\eta_3,\zeta_3$.
Consider the toric fans $\Sigma_2$ and $\Sigma_3$ of the economic resolutions for the type $\frac{1}{7}(1,2,5)$ and the type $\frac{1}{5}(1,2,3)$, respectively.
\subsubsection*{\bf $G$-bricks}\label{subsec:G-bricks for 1/12(1,7,5)}
We now calculate $G$-bricks corresponding to the following two maximal cones in $\Sigma$:
\[
\begin{array}{ll}
\sigma_4 &=\Cone \big( \frac{1}{12}(12,0,0), \frac{1}{12}(3,9,3), \frac{1}{12}(8,8,4) \big),\\[3pt]
\tau_3 &=\Cone \big( \frac{1}{12}(1,7,5), \frac{1}{12}(3,9,3), \frac{1}{12}(8,8,4) \big).
\end{array}
\]
The cones $\sigma_4,\tau_3$ are on the right side of the lowest vector $v=\frac{1}{12}(1,7,5)$. Their corresponding cones $\sigma'_4,\tau'_3$ in $\Sigma_3$, respectively, are
\begin{equation}\label{in Eg section:cones for 1/5(1,2,3)}
\begin{array}{ll}
\sigma'_4 &=\Cone\big( \tfrac{1}{5}(5,0,0), \tfrac{1}{5}(1,2,3), \tfrac{1}{5}(1,1,4)\big),\\[2pt]
\tau'_3 &=\Cone\big( \tfrac{1}{5}(0,0,5), \tfrac{1}{5}(1,2,3), \tfrac{1}{5}(1,1,4) \big).
\end{array}
\end{equation}
Observe that the cones $\sigma'_4,\tau'_3$ are on the left side of $\Sigma_3$. 
To use the recursion, let $G_{32}$ be the group of type $\frac{1}{2}(1,1,1)$ with eigencoordinates $\xi_{32},\eta_{32},\zeta_{32}$. Let $\Sigma_{32}$ denote the fan of the economic resolution of the quotient $\C^3/G_{32}$.
In $\Sigma_{32}$, there exist two cones $\sigma''_4,\tau''_3$ corresponding to $\sigma'_4,\tau'_3$, respectively:
\[
\begin{array}{rl}
\sigma''_4 &=\Cone\big( \frac{1}{2}(2,0,0), \frac{1}{2}(0,2,0), \frac{1}{2}(1,1,1)\big),\\[2pt]
\tau''_3 &=\Cone\big( \frac{1}{2}(0,0,2), \frac{1}{2}(0,2,0), \frac{1}{2}(1,1,1) \big).
\end{array}
\]
\begin{figure}[h]
\begin{center}
\begin{tikzpicture}
\coordinate [label=left:$e_3$] (e3) at (9,0);
\coordinate [label=right:$e_2$] (e2) at (12,0);
\foreach \x in {1,2,3,4,5,6,7,8,9,10,11}
    \draw (0.25*\x+9 ,1pt) -- (0.25*\x+9 ,-1pt);
\foreach \y in {1,2,3,4,5,6,7,8,9,10,11}
    \draw (8.97,0.25*\y) -- (9.03,0.25*\y);
    
\coordinate  (v1) at (9.25,1.75);
\draw[fill] (v1) circle [radius=0.01];
\coordinate  (v2) at (9.5,0.5);
\draw[fill] (v2) circle [radius=0.01];
\coordinate  (v3) at (9.75,2.25);
\draw[fill] (v3) circle [radius=0.01];
\coordinate  (v4) at (10,1);
\draw[fill] (v4) circle [radius=0.01];
\coordinate (v5) at (10.25,2.75);
\draw[fill] (v5) circle [radius=0.01];
\coordinate  (v6) at (10.5,1.5);
\draw[fill] (v6) circle [radius=0.01];
\coordinate  (v7) at (10.75,0.25);
\draw[fill] (v7) circle [radius=0.01];
\coordinate (v8) at (11,2);
\draw[fill] (v8) circle [radius=0.01];
\coordinate  (v9) at (11.25,0.75);
\draw[fill] (v9) circle [radius=0.01];
\coordinate (v10) at (11.5,2.5);
\draw[fill] (v10) circle [radius=0.01];
\coordinate  (v11) at (11.75,1.25);
\draw[fill] (v11) circle [radius=0.01];

\draw (e3) -- (e2);
\draw (e3) -- (9,3);
\draw (e2) -- (12,3);
\draw (e3) -- (v1);
\draw (e3) -- (v6);
\draw (v6) -- (v7);
\draw (e3) -- (v7);
\draw (v7) -- (e2);
\draw (v1) -- (v2);
\draw (v2) -- (v3);
\draw (v3) -- (v4);
\draw (v4) -- (v5);
\draw (v5) -- (v6);
\draw (v6) -- (v7);
\draw (v7) -- (v8);
\draw (v7) -- (v11);
\draw (v9) -- (e2);
\draw (v8) -- (v9);
\draw (v9) -- (v10);
\draw (v10) -- (v11);
\draw (v11) -- (e2);
\draw (v4) -- (v7);
\draw (v2) -- (v7);
\foreach \z in {1,2,3,4,5,6,7,8,9,10,11}
\draw (v\z) -- (9+0.25*\z,3);

\draw[->] (8,1.5) -- (8.5,1.5);
\draw[->] (3.5,1.5) -- (4,1.5);
\node [right] at (12.1,2.8) {$\uparrow e_1$};    

\coordinate [label=left:$e_3$] (e35) at (4.5,0);
\coordinate [label=right:$e_2$] (e25) at (7.5,0);
\foreach \x in {1,2,3,4}
    \draw (0.6*\x+4.5 ,1pt) -- (0.6*\x+4.5 ,-1pt);
\foreach \y in {1,2,3,4}
    \draw (4.47,0.6*\y) -- (4.53,0.6*\y);
    
\coordinate  (v15) at (5.1,1.8);
\draw[fill] (v15) circle [radius=0.01];
\coordinate  (v25) at (5.7,0.6);
\draw[fill] (v25) circle [radius=0.01];
\coordinate  (v35) at (6.3,2.4);
\draw[fill] (v35) circle [radius=0.01];
\coordinate  (v45) at (6.9,1.2);
\draw[fill] (v45) circle [radius=0.01];
\draw (e35) -- (e25);
\draw (e35) -- (4.5,3);
\draw (e25) -- (7.5,3);
\foreach \z in {1,2,3,4}
\draw (v\z5) -- (4.5+0.6*\z,3);
\draw (e35) -- (v15);
\draw (e35) -- (v45);
\draw (v25) -- (v15);
\draw (v25) -- (v35);
\draw (v25) -- (e25);
\draw (v35) -- (e25);
\node [right] at (7.6,2.8) {$\uparrow e_1$};  

\coordinate [label=left:$e_3$] (e32) at (0,0);
\coordinate [label=right:$e_2$] (e22) at (3,0);
\foreach \x in {1}
    \draw (1.5*\x,1pt) -- (1.5*\x ,-1pt);
\foreach \y in {1}
    \draw (-0.03,1.5*\y) -- (0.03,1.5*\y);
    
\node [right] at (3.1,2.8) {$\uparrow e_1$};  
\coordinate  (v12) at (1.5,1.5);
\draw[fill] (v12) circle [radius=0.01];
\draw (e32) -- (e22);
\draw (e32) -- (0,3);
\draw (e22) -- (3,3);
\foreach \z in {1}
\draw (v\z2) -- (1.5*\z,3);
\draw (v12) -- (e32);
\draw (v12) -- (e22);

\node [right] at (2,2) {$\sigma''_{4}$};
\node [right] at (1.2,0.6) {$\tau''_{3}$};    
\node [right] at (5.15,2) {$\sigma'_{4}$};
\node [right] at (4.75,0.8) {$\tau'_{3}$};
    
\node [right] at (0.7,-0.5) {$\tfrac{1}{2}(1,1,1)$};
\node [right] at (5.2,-0.5) {$\tfrac{1}{5}(1,2,3)$};
\node [right] at (9.7,-0.5) {$\tfrac{1}{12}(1,7,5)$};
\end{tikzpicture}
\end{center}
\caption{Recursion process for $\frac{1}{12}(1,7,5)$}\label{Fig:recursion process}
\end{figure}
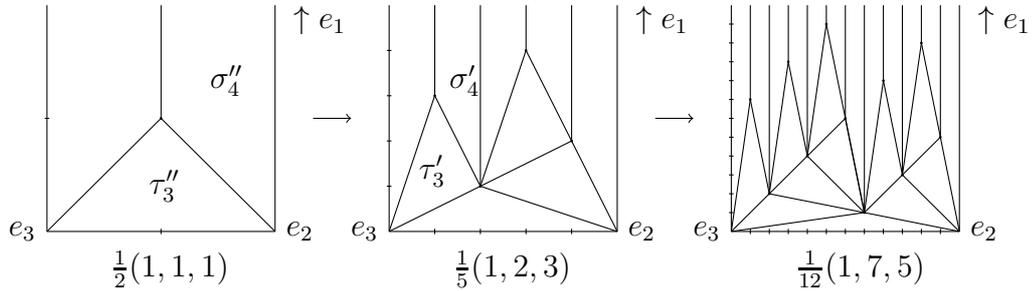

As in Section~\ref{subsec:G-bricks for 1/r(1,r-1,1)}, the $G_{32}$-bricks $\Gamma''_4, \Gamma''_3$ corresponding to $\sigma''_4,\tau''_3$ are
\[
\begin{array}{ll}
\Gamma''_4 &=\big\{1, \zeta_{23} \big\},\\[3pt]
\Gamma''_3 &=\big\{1, \xi_{23}\big\}.
\end{array}
\]
Using the left round down function $\phi_{32}$ for $\frac{1}{5}(1,2,3)$
\[
\phi_{32}\colon \xi_{3}^{a}\eta_{3}^{b}\zeta_{3}^{c}\ \mapsto \ \xi_{32}^{a}\eta_{32}^{\lf\frac{a+2b+3c}{5}\rf}\zeta_{32}^{c},
\]
we can see that the $G_{3}$-bricks $\Gamma'_4, \Gamma'_3$ corresponding to $\sigma'_4,\tau'_3$ are
\[
\begin{array}{lll}
\Gamma'_4 &\eqd\phi_{32}\inv\big(\Gamma''_4\big) &=\big\{1,\, \eta_3, \,\eta_3^2,\,\zeta_3, \,\frac{\zeta_3}{\eta_{3}} \big\},\\[4pt]
\Gamma'_3 &\eqd\phi_{32}\inv\big(\Gamma''_2\big) &=\big\{1,\, \eta_3,\, \eta_3^2,\, \xi_3,\, \xi_3\eta_3 \big\}.
\end{array}
\]
To get the $G$-bricks $\Gamma_4$ and $\Gamma_3$ corresponding to $\sigma_4$ and $\tau_3$, respectively,
we use the right round down function $\phi_{3}$ for $\frac{1}{12}(1,7,5)$:
\[
\phi_{3}\colon x^{a}y^{b}z^{c}\ \mapsto \ \xi_{3}^{a}\eta_{3}^{b}\zeta_{3}^{\lf\frac{a+7b+5c}{12}\rf}.
\]
We get
\[
\begin{array}{lll}
\Gamma_4 &\eqd\phi_{3}\inv\large(\Gamma'_4\large) &=\Big\{1, y, \frac{y}{z},\frac{y^2}{z}, \frac{y^2}{z^2}, z, z^2,z^3,z^4,\frac{z^4}{y},\frac{z^5}{y},\frac{z^6}{y}  \Big\},\\[4pt]
\Gamma_3 &\eqd\phi_{3}\inv\large(\Gamma'_2\large) &=\left\{1,  x,xz,xz^2, xy,\frac{xy}{z}, y, \frac{y}{z},\frac{y^2}{z}, \frac{y^2}{z^2}, z, z^2 \right\}.
\end{array}
\]

Let us consider the following two cones in $\Sigma$:
\begin{align*}
\sigma_9 &=\Cone\big( \tfrac{1}{12}(12,0,0), \tfrac{1}{12}(9,3,9), \tfrac{1}{12}(4,4,8)\big),\\[3pt]
\tau_7 &=\Cone\big( \tfrac{1}{12}(2,2,10), \tfrac{1}{12}(9,3,9), \tfrac{1}{12}(4,4,8) \big).
\end{align*}
Observe that the cones $\sigma_9$, $\tau_7$ are on the left side of $v$. The cones in $\Sigma_2$ corresponding to $\sigma_9$, $\tau_7$ are
\begin{align*}
\sigma'_9 &=\Cone\big( \tfrac{1}{7}(12,0,0), \tfrac{1}{7}(5,3,4), \tfrac{1}{7}(2,4,3)\big),\\[3pt]
\tau'_7 &=\Cone\big( \tfrac{1}{7}(1,2,5), \tfrac{1}{7}(5,3,4), \tfrac{1}{7}(2,4,3) \big).
\end{align*}
Note that the cones $\sigma'_9$, $\tau'_7$ are on the right side of the fan $\Sigma_2$ and that the right side is equal to the fan $\Sigma_3$ of the economic resolution for $\frac{1}{5}(1,2,3)$. Moreover, the cones in $\Sigma_3$ corresponding to $\sigma'_9$, $\tau'_7$ are $\sigma'_4$, $\tau'_3$, respectively, in \eqref{in Eg section:cones for 1/5(1,2,3)}. Thus the corresponding $G_{23}$-bricks $\Gamma''_9$, $\Gamma''_7$ are:
\begin{align*}
\Gamma''_9  &=\big\{1,\, \eta_{23}, \,\eta_{23}^2,\,\zeta_{23}, \,\tfrac{\zeta_{23}}{\eta_{23}} \big\},\\[3pt]
\Gamma''_7 &=\big\{1,\, \xi_{23},\, \xi_{23}\eta_{23},\, \eta_{23},\, \eta_{23}^2 \big\},
\end{align*}
where $G_{23}$ denotes the group of type $\frac{1}{5}(1,2,3)$ with eigencoordinates $\xi_{23},\eta_{23},\zeta_{23}$.
Using the right round down function $\phi_{23}$ for $\frac{1}{7}(1,2,5)$
\[
\phi_{23}\colon \xi_{2}^{a}\eta_{2}^{b}\zeta_{2}^{c}\ \mapsto \ \xi_{23}^{a}\eta_{23}^{b}\zeta_{23}^{\lf\frac{a+2b+5c}{7}\rf},
\]
we can calculate the $G_2$-bricks corresponding to $\sigma'_9,\tau_7'$:
\[
\begin{array}{lll}
\Gamma'_9 &\eqd\phi_{23}\inv\large(\Gamma''_9\large) &=\Big\{1,\, \eta_2,\,\eta_2^2,\,\zeta_2,\,\zeta_2^2,\,\tfrac{\zeta_2^2}{\eta_2},\, \tfrac{\zeta_2^3}{\eta_2}\Big\},\\[6pt]
\Gamma'_7 &\eqd\phi_{23}\inv\large(\Gamma''_7\large) &=\Big\{ 1, \,\xi_2, \, \xi_2\eta_2,\,\xi_2\zeta_2,\, \eta_2,\,\eta_2^2,\,\zeta_2,\,\zeta_2^2\Big\}.
\end{array}
\]
Lastly, from the left round down function $\phi_2$ for $\frac{1}{12}(1,7,5)$
\[
\phi_{2}\colon x^{a}y^{b}z^{c}\ \mapsto \ \xi_{2}^{a}\eta_{2}^{\lf\frac{a+7b+5c}{12}\rf}\zeta_{2}^{c},
\]
it follows that the $G$-bricks $\Gamma_9,\Gamma_7$ corresponding to $\sigma_9, \tau_7$ are:
\begin{align*}
\Gamma_9 &=\Big\{1,y,y^2,y^3,y^4,y^5,z,z^2,\tfrac{z^2}{y},\tfrac{z^2}{y^2}, \tfrac{z^2}{y^3}, \tfrac{z^3}{y^3}\Big\},\\[5pt]
\Gamma_7 &=\Big\{1,x,xy,xy^2,xy^3,xz, y,y^2,y^3,y^4,y^5,z\Big\}.
\end{align*}

For $0\leq i\leq 12$, let $v_i$ denote the lattice point $\frac{1}{12}(\overline{7i},i,12-i)$ in $L$.
For each 3-dimensional cone $\sigma$ in Figure~\ref{Fig:Toric Fan of 1/12(1,7,5)} on page~\pageref{Fig:Toric Fan of 1/12(1,7,5)}, Table~\ref{Tab:G-iraffes for 1/12(1,7,5)} on page~\pageref{Tab:G-iraffes for 1/12(1,7,5)} shows the corresponding $G$-brick $\Gamma_{\sigma}$.
\begin{table}
\begin{center}
\begin{tabular}{c|c|c|> {\cellstart}m{6em}<{\cellfinish}}\hline
Cone& Generators & $G$-brick $\Gamma_\sigma$ & Coordinates on $U_{\sigma}$ \\ \hline\hline

$\sigma_1$ &$e_1,e_2, v_{11}$& $1,z^{},z^{2},z^{3},z^{4},z^{5},z^{6},z^{7},z^{8},z^{9},z^{10},z^{11}$ & $\tfrac{x}{z^5},\tfrac{y^{}}{z^{11}},{z^{12}}$\\ \hline

$\sigma_{2}$ &$e_1,v_{10},v_{11}$& $1,y,\tfrac{y}{z},z^{},z^{2},z^{3},z^{4},z^{5},z^{6},z^{7},z^{8},z^{9}$ & $\tfrac{x}{z^5},\tfrac{y^{2}}{z^{10}},\tfrac{z^{11}}{y^{}}$\\ \hline

$\sigma_{3}$ &$e_1,v_{9},v_{10}$& $1, y, \tfrac{y}{z},\tfrac{y^2}{z}, \tfrac{y^2}{z^2},\tfrac{y^2}{z^3},\tfrac{y^2}{z^4}, \tfrac{y^2}{z^5},z^{},z^{2},z^{3},z^{4}$ & $\tfrac{xz^5}{y^2},\tfrac{y^{3}}{z^{9}},\tfrac{z^{10}}{y^{2}}$\\ \hline

$\sigma_{4}$ &$e_1,v_{8},v_{9}$& $1, y, \frac{y}{z},\frac{y^2}{z}, \frac{y^2}{z^2}, z, z^2,z^3,z^4,\frac{z^4}{y},\frac{z^5}{y},\frac{z^6}{y}$ & $\tfrac{xy}{z^4},\tfrac{y^{4}}{z^{8}},\tfrac{z^{9}}{y^{3}}$\\ \hline

$\sigma_{5}$ &$e_1,v_{7},v_{8}$& $1, y, \tfrac{y}{z}, \tfrac{y^2}{z}, \tfrac{y^2}{z^2}, \tfrac{y^3}{z^2},\tfrac{y^3}{z^3}, \tfrac{y^3}{z^4},\tfrac{y^4}{z^4},\tfrac{y^4}{z^5},z,z^2$ & $\tfrac{xz^4}{y^3},\tfrac{y^{5}}{z^{7}},\tfrac{z^{8}}{y^{4}}$\\ \hline

$\sigma_{6}$ &$e_1,v_{6},v_{7}$& $1,y,z,z^2, \tfrac{z^2}{y}, \tfrac{z^3}{y}, \tfrac{z^3}{y^2}, \tfrac{z^4}{y^2}, \tfrac{z^5}{y^2}, \tfrac{z^5}{y^3}, \tfrac{z^6}{y^3}, \tfrac{z^6}{y^4}$ & $\tfrac{xy^2}{z^3},\tfrac{y^{6}}{z^{6}},\tfrac{z^{7}}{y^{5}}$\\ \hline

$\sigma_{7}$ &$e_1,v_{5},v_{6}$& $1, y, y^2,y^3,z,z^2, \tfrac{z^2}{y}, \tfrac{z^3}{y}, \tfrac{z^3}{y^2}, \tfrac{z^4}{y^2}, \tfrac{z^5}{y^2}, \tfrac{z^5}{y^3}$ & $\tfrac{xy^2}{z^3},\tfrac{y^{7}}{z^{5}},\tfrac{z^{6}}{y^{6}}$\\ \hline

$\sigma_{8}$ &$e_1,v_{4},v_{5}$& $1, y,y^2,y^3,y^4,y^5,\tfrac{y^5}{z},\tfrac{y^5}{z^2},\tfrac{y^6}{z^2},z,z^2,\tfrac{z^2}{y}$ & $\tfrac{xz^2}{y^5},\tfrac{y^{8}}{z^{4}},\tfrac{z^{5}}{y^{7}}$ \\ \hline

$\sigma_{9}$ &$e_1,v_{3},v_{4}$& $1,y,y^2,y^3,y^4,y^5,z,z^2,\tfrac{z^2}{y},\tfrac{z^2}{y^2}, \tfrac{z^2}{y^3}, \tfrac{z^3}{y^3}$ & $\tfrac{xy^3}{z^2},\tfrac{y^{9}}{z^{3}},\tfrac{z^{4}}{y^{8}}$\\ \hline

$\sigma_{10}$ &$e_1,v_{2},v_{3}$& $1,y,y^2,y^3,y^4,y^5,y^6,\tfrac{y^6}{z},\tfrac{y^7}{z},\tfrac{y^8}{z},\tfrac{y^9}{z},z$ & $\tfrac{xz}{y^6},\tfrac{y^{10}}{z^{2}},\tfrac{z^{3}}{y^{9}}$\\ \hline

$\sigma_{11}$ &$e_1,v_{1},v_{2}$& $1,y,y^2,y^3,y^4,y^5,y^6,z,\tfrac{z}{y},\tfrac{z}{y^2},\tfrac{z}{y^3},\tfrac{z}{y^4}$ & $\tfrac{xy^4}{z},\tfrac{y^{11}}{z^{1}},\tfrac{z^{2}}{y^{10}}$\\ \hline

$\sigma_{12}$ &$e_1,e_3,v_{1}$& $1,y,y^{2},y^{3},y^{4},y^{5},y^{6},y^{7},y^{8},y^{9},y^{10},y^{11}$ & $\tfrac{x}{y^7},{y^{12}},\tfrac{z^{}}{y^{11}}$ \\ \hline

$\tau_{1}$ &$e_2,v_{9},v_{11}$& $1,x,xz,xz^2,xz^3,xz^4, x^2,x^2z,z,z^2,z^3,z^4$ &$\tfrac{x^3}{z^3},\tfrac{y}{x^2z},\tfrac{z^5}{x}$ \\ \hline

$\tau_{2}$ &$v_{9},v_{10},v_{11}$& $1,x,z,xz,z^2,xz^2,z^3,xz^3,z^4,xz^4,y,\tfrac{y}{z}$ &$\tfrac{x^2z}{y},\tfrac{y^2}{xz^5},\tfrac{z^5}{x}$ \\ \hline

$\tau_{3}$ &$v_{7},v_{8},v_{9}$& $1, x,xy,\frac{xy}{z},xz,xz^2, y, \frac{y}{z},\frac{y^2}{z}, \frac{y^2}{z^2}, z, z^2$ &$\tfrac{x^2z}{y},\tfrac{y^3}{xz^4},\tfrac{z^4}{xy}$ \\ \hline

$\tau_{4}$ &$e_{2},v_{7},v_{9}$& $1,x,x^2,x^3,x^4,xz,xz^2,x^2z,x^3z,x^4z,z,z^2$ & $\tfrac{x^5}{z},\tfrac{y}{x^2z},\tfrac{z^3}{x^3}$\\ \hline

$\tau_{5}$ &$v_{4},v_{6},v_{7}$& $1,x,xy,xz,xz^2, \tfrac{xz^2}{y},x^2,x^2y,y,z,z^2,\tfrac{z^2}{y}$ & $\tfrac{x^3y}{z^2},\tfrac{y^2}{x^2},\tfrac{z^3}{xy}$\\ \hline

$\tau_{6}$ &$v_{4},v_{5},v_{6}$& $1,x,xy,xz,xz^2,\tfrac{xz^2}{y}, y,y^2,y^3,z,z^2,\tfrac{z^2}{y}$ &$\tfrac{x^2}{y^2},\tfrac{y^5}{xz^2},\tfrac{z^3}{xy^2}$ \\ \hline

$\tau_{7}$ &$v_{2},v_{3},v_{4}$& $1,x,xy,xy^2,xy^3,xz, y,y^2,y^3,y^4,y^5,z $ & $\tfrac{x^2}{y^2},\tfrac{y^6}{xz},\tfrac{z^2}{xy^3}$ \\ \hline

$\tau_{8}$ &$v_{2},v_{4},v_{7}$& $1,x,xy,xz,x^2,x^2y,x^3,x^3y,x^4,x^4y,y,z$ &$\tfrac{x^5}{z},\tfrac{y^2}{x^2},\tfrac{z^2}{x^3y}$ \\ \hline

$\tau_{9}$ &$e_3,v_{1},v_{2}$& $1,x,xy,xy^2,xy^3,xy^4,y,y^2,y^3,y^4,y^5,y^6$ &$\tfrac{x^2}{y^2}, \tfrac{y^7}{x},\tfrac{z}{xy^4} $ \\ \hline

$\tau_{10}$ &$e_3,v_{2},v_{7}$& $1,x,xy,x^2,x^2y,x^3,x^3y,x^4,x^4y,x^5,x^6,y$ & $\tfrac{x^7}{y},\tfrac{y^2}{x^2},\tfrac{z}{x^5} $\\ \hline

$\tau_{0}$ &$e_{2},e_{3},v_{7}$&  $1,x,x^2,x^3,x^4,x^5,x^6,x^7,x^8,x^9,x^{10},x^{11}$ &$x^{12},\tfrac{y}{x^7}, \tfrac{z}{x^5}$ \\ \hline
\end{tabular}
\end{center}
\caption{$G$-bricks for $G=\frac{1}{12}(1,7,5)$}\label{Tab:G-iraffes for 1/12(1,7,5)}
\end{table}
\subsubsection*{\bf K\k{e}dzierski's GIT chamber}\label{subsec:Kedzierski's GIT chamber for 1/12(1,7,5)}
We calculate K\k{e}dzierski's GIT chamber for $\frac{1}{12}(1,7,5)$.
Since the economic resolution is $\GHil$ for the group of type $\frac{1}{r}(1,r-1,1)$, the sets of simple roots for $\frac{1}{2}(1,1,1)$ and $\frac{1}{3}(1,2,1)$ are $\{\varepsilon_1-\varepsilon_0\}, \{\varepsilon_1-\varepsilon_2, \varepsilon_2-\varepsilon_0\}$,
respectively. By the identification~\eqref{Eqtn:identification of basis}, the set of simple roots for 
$\frac{1}{5}(1,2,3)$ is
\[
\{\varepsilon_3-\varepsilon_4,\ \underline{\varepsilon_4-\varepsilon_1},\ \varepsilon_1-\varepsilon_2,\ \varepsilon_2-\varepsilon_0\},
\]
where the underlined root is the added root as in~\eqref{Eqtn:Def.Admissible roots}. Similarly, the admissible set of simple roots for 
$\frac{1}{7}(1,2,5)$ is
\[
\{
\varepsilon_5-\varepsilon_6,\ \underline{\varepsilon_6-\varepsilon_3},\ \varepsilon_3-\varepsilon_4,\ {\varepsilon_4-\varepsilon_1},\ \varepsilon_1-\varepsilon_2,\ \varepsilon_2-\varepsilon_0\}.
\]
Thus the corresponding set of simple roots for $\frac{1}{12}(1,7,5)$ is
\[
\left\{
\begin{matrix}
\varepsilon_5-\varepsilon_6,\ \varepsilon_6-\varepsilon_{10}, \
\varepsilon_{10}-\varepsilon_{11},\ \varepsilon_{11}-\varepsilon_{8}, \
\varepsilon_8-\varepsilon_9,\ \varepsilon_9-\varepsilon_7,
\\
\underline{\varepsilon_7-\varepsilon_3}, \
\varepsilon_3-\varepsilon_4, \ {\varepsilon_4-\varepsilon_1},\
\varepsilon_1-\varepsilon_2,\ \varepsilon_2-\varepsilon_0
\end{matrix}
\right\}.
\]

With the dual basis $\{\theta_i\}$ with respect to $\{\rho_i\}$, the row vectors of the following matrix are the rays of the admissible Weyl chamber $\wca$:
\[
\left(
\begin{array}{rrrrrrrrrrrr}
-1 & 0 & 0 & 0   & 0 & 0 & 0 & 1   & 0 & 0 & 0 & 0 \\
-1 & 0 & -1 & 0  & 0 & 0 & 0 & 1  & 0 & 1 & 0 & 0 \\
-1 & -1 & -1 & 0 & 0 & 0 & 0 & 1  & 1 & 1 & 0 & 0 \\
-1 & -1 & -1 & 0 & -1 & 0 & 0 & 1  & 1 & 1 & 0 & 1 \\
-1 & -1 & -1 & -1 & -1 & 0 & 0 & 1  & 1 & 1 & 1 & 1 \\
-1 & -1 & 0 & -1 & -1 & 0 & 0 & 0  & 1 & 1 & 1 & 1 \\
-1 & -1 & 0 & -1 & 0 & 0 & 0 & 0  & 1 & 0 & 1 & 1 \\
-1 & -1 & 0 & 0 & 0 & 0 & 0 & 0  & 0 & 0 & 1 & 1 \\
-1 & -1 & 0 & 0 & 0 & 0 & 1 & 0  & 0 & 0 & 1 & 0 \\
-1 & -1 & 0 & 0 & 0 & 1 & 1 & 0  & 0 & 0 & 0 & 0 \\
-1 &  0 & 0 & 0 & 0 & 1 & 0 & 0  & 0 & 0 & 0 & 0 \\
\end{array}
\right).
\]

\bibliographystyle{alpha}


\end{document}